\documentclass[11pt,reqno]{amsart}
\usepackage{amsmath, amssymb, amsthm}
\usepackage{url}
\usepackage[breaklinks]{hyperref}
\usepackage{mathrsfs}

\newcommand{\s}{{\sigma}}

 \renewcommand{\a}{\alpha}
\renewcommand{\b}{\beta}

\newcommand{\g}{\gamma}
\newcommand{\G}{\Gamma}
\renewcommand{\l}{\lambda}

\renewcommand{\(}{\left\(}
\renewcommand{\)}{\right\)}
\renewcommand{\[}{\left\[}
\renewcommand{\]}{\right\]}
\numberwithin{equation}{section}
 \theoremstyle{plain}
\newtheorem{theorem}{Theorem}[section]
\newtheorem{lemma}[theorem]{Lemma}

\newtheorem{corollary}[theorem]{Corollary}

   \makeatletter
\def\proof{\@ifnextchar[{\@oproof}{\@nproof}}
\def\@oproof[#1][#2]{\trivlist\item[\hskip\labelsep\textit{#2 Proof of\
#1.}~]\ignorespaces}
\def\@nproof{\trivlist\item[\hskip\labelsep\textit{Proof.}~]\ignorespaces}

\makeatother

\begin{document}
\title[]{Koshliakov zeta functions I: Modular Relations } 
\dedicatory{In memory of Nikolai Sergeevich Koshliakov}
\author{Atul Dixit, Rajat Gupta}
\address{Discipline of Mathematics, Indian Institute of Technology Gandhinagar, Palaj, Gandhinagar 382355, Gujarat, India} 
\email{adixit@iitgn.ac.in, ~rajat\_gupta@iitgn.ac.in }
\begin{abstract}
We examine an unstudied manuscript of N.~S.~Koshliakov over $150$ pages long and containing the theory of two interesting generalizations $\zeta_p(s)$ and $\eta_p(s)$ of the Riemann zeta function $\zeta(s)$, which we call \emph{Koshliakov zeta functions}. His theory has its genesis in a problem in the analytical theory of heat distribution which was analyzed by him. In this paper, we further build upon his theory and obtain two new modular relations in the setting of Koshliakov zeta functions, each of which gives an infinite family of identities, one for each $p\in\mathbb{R^{+}}$. The first one is a generalization of Ramanujan's famous formula for $\zeta(2m+1)$ and the second is an elegant extension of a modular relation on page $220$ of Ramanujan's Lost Notebook. Several interesting corollaries and applications of these modular relations are obtained including a new representation for $\zeta(4m+3)$.
\end{abstract}
\maketitle
\tableofcontents
\section{An unstudied manuscript of N.~S.~Koshliakov}
The discovery of Ramanujan's Lost Notebook, which is a 138-page manuscript (along with some loose papers) which Ramanujan worked on in the last year of his life, has been considered to be the mathematical equivalent of the discovery of Beethoven's tenth symphony \cite{AB5}. Ever since George E.~Andrews discovered the Lost Notebook in 1976, there have been hundreds of papers written on the formulas given in there spanning variegated areas of mathematics such as analytic number theory, $q$-series, theory of partitions, modular and mock modular forms, mathematical physics, to name a few. The charming story of the discovery of the Lost Notebook is well-known, see, for example, \cite{discoveryrln}.

In this paper, we bring to light a manuscript which fell through the cracks during the aftermath of World War II. This manuscript was written by Nikolai Sergeevich Koshliakov, an outstanding Russian mathematician, who made seminal contributions to analytic number theory and differential equations. An interesting account of the adverse conditions in which it was written and how it became available to the mathematical community is documented in the article \cite{bfikm} (see also \cite{lady}). Since it is extremely inspiring, we reproduce it below so as to make it available to the broader mathematical community. 

\begin{quote} The repressions of the thirties which affected scholars in Leningrad continued even after the outbreak of the Second World War. In the winter of 1942 at the height of the blockade of Leningrad, Koshlyakov along with a group \dots was arrested on fabricated \dots dossiers and condemned to 10 years correctional hard labour.  After the verdict he was exiled to one of the camps in the Urals. \dots On the grounds of complete exhaustion and complicated pellagra, Koshlyakov was classified in the camp as an invalid and was not sent to do any of the usual jobs. \dots very serious shortage of paper.  He was forced to carry out calculations on a piece of plywood, periodically scraping off what he had written with a piece of glass. Nevertheless, between 1943 and 1944 Koshlyakov wrote two long memoirs \emph{Issledovanie nekotorykh voprosov analyticheskoi teorii rational'nogo i
kvadratichnogo polya} (A study of some questions in the analytic theory of
rational and quadratic fields) and \emph{Issledovanie odnogo klassa transtsendentnykh
funktsii, opredelyaemykh obobshchennym yravneniem Rimana} (A study of a class
of transcendental functions defined by the generalized Riemann equation). \end{quote}

The first memoir in the above paragraph was lost in transit from jail to the mathematical community as mentioned in \cite{web}. However, we believe that Koshliakov reproduced the work in the memoir in the three papers \cite{koshliakov6} he wrote after he was released from jail. 

The second memoir, which was written under his patronymic name (N. S. Sergeev), is the subject of discussion of our paper. I.~M.~Vinogradov, S.~Bernstein and Yu.~V.~Linnik was so impressed by this manuscript that they immediately recommended its publication. Sadly though, this manuscript has never been examined in detail since then. There is only a very brief mention of it in a PhD thesis of A.~G.~Kisunko \cite[p.~4]{kisunko} and it can also be found in some news articles on the web, for example, in \cite{webarticle}. The first author found out about it first in 2010 while reading \cite{bfikm} and kindly got a copy of it from the Center for Research Libraries in Chicago through the inter-library loan service of University of Illinois at Urbana-Champaign.

But it was not until last year that we started studying this manuscript. Having examined it detail since then has convinced us that it is indeed a masterpiece! 

In a series of papers beginning with this one, we plan to further build upon Koshliakov's theory. This theory gives as a special case the theory of the Riemann zeta function $\zeta(s)$. In this first paper, we obtain new generalized modular relations using Koshliakov's theory which result in two of Ramanujan's famous formulas as special cases. The surprising thing is that our generalized modular relations give new results even in the theory of the Riemann zeta function. These results are given in Section \ref{mr}.

\section{Contents of Koshliakov's manuscript}\label{contents}
Koshliakov's manuscript is over $150$ pages long. Even though he worked on it in 1940s while in prison, the churning of the concepts in this manuscript seems to be going on in his head right from mid-1930s as can be seen from the articles \cite{koshmathann} and \cite{koshbernoulli}. The origin of his work stems from a problem in physics arising in heat conduction \cite{koshbernoulli} which we now describe. 

Consider a sphere of radius $R_2$. Suppose the radiation occurs on its surface $r=R_2$ into a medium at the zero temperature. Suppose the initial temperature is $u=0$, the heat sources are placed on a spherical surface of radius $R_1$\hspace{1mm} $(0<R_1<R_2)$, and the rate per unit time of heat propagation from the whole surface is $Q$. Then the problem is concerned with finding the temperature of the sphere at time $t>0$. If $k$ is the thermal conductivity and $h$ is the emissivity of the surface, if the specific heat is $c$ and $\rho$ is the density of the material forming the sphere, then the relevant form of the heat equation turns out to be
\begin{align*}\label{he}
\frac{\partial v}{\partial t}&=a^2\frac{\partial^2 v}{\partial r^2}, a=\sqrt{c/(k\rho)};\nonumber\\
v|_{r=0}&=0, \left.\frac{\partial v}{\partial r}\right|_{r=R_2}+\left.\left(H-\frac{1}{R_2}\right)v\right|_{r=R_2}=0, H=h/k.
\end{align*}
Such problems are very important in the analytical theory of heat distribution. See \cite{betsoly} for a recent article on the steady state distribution of heat.

It can be easily verified by the method of separation of variables that the characteristic solution to the above system is
\begin{equation}\label{ce1}
\mu\cos\mu+p\sin\mu=0,
\end{equation}
where $p=R_2H-1$ and $\mu$ is its eigenvalue. It is this characteristic equation that forms the crux of Koshliakov's theory of transcendental functions defined by a generalization of the functional equation of the Riemann zeta function. This is now explained.

The Riemann zeta function $\zeta(s)$ is defined for Re $s>1$ by the absolutely convergent Dirichlet series $\zeta(s)=\sum_{m=1}^{\infty}m^{-s}$. It can be analytically continued to the whole complex plane except for a simple pole at $s=1$. At the heart of its theory lies the functional equation for $\zeta(s)$ obtained by Riemann himself. It is given by \cite[p.~16, Equation (2.1.13)]{titch}
\begin{equation}\label{zetafesym}
\pi^{-\frac{s}{2}}\Gamma\left(\frac{s}{2}\right)\zeta(s)=\pi^{-\frac{(1-s)}{2}}\Gamma\left(\frac{1-s}{2}\right)\zeta(1-s),\nonumber
\end{equation}
which can also be written in the following asymmetric form \cite[p.~13, Equation (2.1.1)]{titch}:
\begin{equation}\label{zetafe}
\zeta(1-s)=2(2\pi)^{-s}\G(s)\zeta(s)\cos\left(\frac{\pi s}{2}\right).
\end{equation}
In \cite{ham2}, Hamburger obtained the following characterization of the Riemann zeta function:

\textit{If the Dirichlet series
\begin{align}
f(s):=\sum_{n=1}^{\infty}\frac{a(n)}{n^s},\hspace{2mm}g(s):=\sum_{n=1}^{\infty}\frac{b(n)}{n^s},
\end{align}
are absolutely convergent for large $\textup{Re}(s)$, and if $f(s)$ is a meromorphic function of finite order with finitely many poles satisfying
\begin{align*}
\pi^{-s/2}\Gamma\left(\frac{s}{2}\right)f(s)=\pi^{-(1-s)/2}\Gamma\left(\frac{1-s}{2}\right)g(1-s),
\end{align*}
then, $f(s)=g(s)=c\zeta(s)$, where $c$ is a constant.} 

Along with \cite{ham2}, Hamburger \cite{ham1} \cite{ham4} and \cite{ham5} also studied the analytic properties of functions defined by the following generalization of \eqref{zetafe}, namely,
\begin{align}\label{hamburgerfe}
f(1-s)=\frac{2 \cos\left(\frac{\pi s}{2}\right)\Gamma(s)}{(2\pi)^s}g(s).
\end{align}
Here, $f(s)$ is a Dirichlet series whereas $g(s)$, in general, is not a Dirichlet series. However, he neither gave any concrete examples for constructing functions defined by the above series nor examined the study of special functions associated to these series in the same way that the Euler gamma function, Bernoulli numbers and polynomials are associated with $\zeta(s)$. 

The first such pair of $f$ and $g$ was given by Koshliakov in \cite{koshliakov3}.
Note that Hurwitz's formula, valid\footnote{It is also valid for Re$(s)>0$ if $0<a<1$.} for Re$(s)>1$ and $0<a\leq1$, is given by
\begin{align*}
\zeta(1-s, a)=\frac{2\cos\left(\frac{\pi s}{2}\right)\G(s)}{(2\pi)^s}\sum_{n=1}^{\infty}\frac{\left(\cos(2\pi na)+\tan\left(\frac{\pi s}{2}\right)\sin(2\pi na)\right)}{n^s}.
\end{align*}
However, even though the left-hand side in this formula is a Dirichlet series and the right-hand side is not, it is not a valid example as the above relation is not true for \emph{all} complex values of $s$.

For $\textup{Re}(s)>1$, the $f$ in \eqref{hamburgerfe} considered by Koshliakov \cite{koshliakov3} is a more general Dirichlet series than the usual $\sum_{n=1}^{\infty}a(n)n^{-s}$, namely,
\begin{align}\label{zps}
\zeta_p(s):=\sum_{j=1}^{\infty}\frac{p^2+\lambda^2_{j}}{p\left(p+\frac{1}{\pi}\right)+\lambda^2_j}.\frac{1}{\lambda^s_j},
\end{align}
where  $\lambda_1, \lambda_2, ...$ are  the positive roots of the equation
\begin{align}\label{ce2}
p \sin(\pi \lambda)+\lambda \cos(\pi \lambda)=0 \hspace{10mm}(p>0).
\end{align}
Note that the equation \eqref{ce2} is the same as in \eqref{ce1} as can be seen by replacing in \eqref{ce1} $\mu$ and $p$ by $\pi\l$ and $\pi p$ respectively. In his manuscript \cite[p.~15-16, Chapter 1]{koshliakov3}, he proves the absolute and uniform convergence of the series in \eqref{zps} for $\textup{Re}(s)>1$, whence it is seen that $\zeta_p(s)$ is analytic for Re$(s)>1$.

With the choice of $f$ in \eqref{hamburgerfe} being $\zeta_p(s)$ from \eqref{zps}, Koshliakov gets $g$ to be the following series:
\begin{equation}\label{eps}
\eta_{p}(s):= \sum_{k=1}^{\infty}\frac{(s,2 \pi p k)_k}{k^s}\hspace{5mm}(\textup{Re}(s)>1),
\end{equation}
where
\begin{equation}\label{s}
(s,\nu k)_k:=\frac{1}{\Gamma(s)}\int_{0}^{\infty}e^{-x}\left(\frac{k\nu-x}{k\nu+x}\right)^kx^{s-1}\, dx.
\end{equation}
The derivation for the same is given in Sections 2 and 3 of Chapter 1 of \cite[p.~15-20]{koshliakov3}. On page 20 of his manuscript \cite{koshliakov3}, Koshliakov discusses the absolute and uniform convergence of the series defining $\eta_p(s)$ in Re$(s)>1$, which is easily seen since $\frac{2\pi pk-x}{2\pi pk+x}<1$.
Hence $\eta_p(s)$ is also an analytic function of $s$ in Re$(s)>1$.

We call the functions $\zeta_p(s)$ and $\eta_p(s)$ defined in \eqref{zps} and \eqref{eps} as the \emph{Koshliakov zeta functions}\footnote{Even though the series definition of $\eta_p(s)$ is not a Dirichlet series, we call $\eta_p(s)$ a zeta function since it will be always studied in conjunction with $\zeta_p(s)$ which is, indeed, a Dirichlet series.}. They both reduce to the Riemann zeta function $\zeta(s)$ when we let $p\to\infty$. This can be seen as follows. From \eqref{ce2}, we have
\begin{align}
\sin(\pi \lambda)+\frac{\lambda}{p} \cos(\pi \lambda)=0.\nonumber
\end{align}
Hence when $p\to\infty$, the roots of the above equation simply turn out to be positive integers, that is, $\l_j\to j$, so that 
\begin{equation}\label{zpinf}
\lim_{p\to\infty}\zeta_{p}(s)=\zeta(s).
\end{equation}
Also, from \eqref{s}, $\lim_{p\to\infty}(s, 2\pi pk)_k=1$, whence
\begin{equation}
\lim_{p\to\infty}\eta_{p}(s)=\zeta(s).\nonumber
\end{equation}

Another important special case of the Koshliakov zeta functions arises when we let $p\to 0$. From \eqref{ce2}, it is clear that $\l_j\to j-1/2, j\geq1$. Then, by elementary arguments,
\begin{equation}\label{zpzero}
\lim_{p\to0}\zeta_{p}(s)=(2^s-1)\zeta(s).
\end{equation}
Moreover, from \eqref{s}, $(s, 0)_k=(-1)^k$, and hence
\begin{equation}\label{epzero}
\lim_{p\to0}\eta_{p}(s)=(2^{1-s}-1)\zeta(s).
\end{equation}
Let $0<\a<\l_1$, where $\l_1$ is the smallest positive root of \eqref{ce2}. Then Koshliakov obtains the analytic continuation of $\zeta_p(s)$ in the entire complex plane except for a simple pole at $s=1$ with residue $1$ by obtaining the representation \cite[p.~17, Chapter 1, Equation (16)]{koshliakov3}
\begin{align}
\zeta_{p}(s)=\frac{\a^{1-s}}{s-1}+\int_{\a}^{\a-i \infty}\frac{z^{-s}}{\sigma(iz)e^{2\pi iz}-1}dz+\int_{\a}^{\a+i \infty}\frac{z^{-s}}{\sigma(-iz)e^{-2\pi iz}-1}dz,\nonumber
\end{align}
where
\begin{align}\label{sigmap}
\sigma(z)=\frac{p+z}{p-z}.
\end{align}
The other zeta function of Koshliakov, namely, $\eta_p(s)$, can also be analytically continued in the whole $s$-complex plane except for a simple pole at $s=1$ with residue $\frac{1}{1+1/(\pi p)}$ as can be seen in \cite[p.~21-22, Chapter 1]{koshliakov3}.

The functional equation that the Koshliakov zeta functions satisfy is \cite[p.~20, Chapter 1, Equation (30)]{koshliakov3}
\begin{align}\label{2.30}
\zeta_{p}(1-s)=\frac{2 \cos\left(\frac{\pi s}{2}\right)\Gamma(s)}{(2\pi)^s}\eta_{p}(s),
\end{align}
which, according to \eqref{zpinf}-\eqref{epzero}, reduces to \eqref{zetafe} when we let $p\to\infty$ or $p\to0$.

Having considered the setup in \eqref{zps}-\eqref{s}, Koshliakov proceeds to construct the complete theory of his zeta functions and that of the functions associated to them, that is, there are generalized gamma functions (which he calls \emph{gammamorphic functions}), generalized Euler constants, generalized Bernoulli numbers and polynomials, generalized theta function, generalized Hurwitz zeta functions, to name a few. There is also a chapter fully devoted to summation formulas where he obtains generalized Abel-Plana summation formula and generalized Poisson summation formula. The last chapter of the manuscript is based on generalizing some beautiful modular relations of Ramanujan \cite{riemann} and Hardy \cite{ghh}. The definitions and locations of some of these generalizations by Koshliakov in his manuscript \cite{koshliakov3} along with the classical functions as special cases are given in a table after Section \ref{cr}.

In this paper, we concentrate on two different kinds of modular relations resulting by developing further the theory of Koshliakov zeta functions. The first kind of modular relations which we study generalizes Ramanujan's famous formula for odd zeta values. The second kind is concerned with an integral containing in its integrand the Riemann's function $\Xi(t)$ (see \eqref{xiXi} below) and its generalization $\Xi_p(t)$ given in \eqref{10.10}. Koshliakov studies some such relations in the last chapter of his manuscript \cite{koshliakov3}, however, the associated integrals in his formulas always contain a single $\Xi_p(t)$ in their integrands whereas our modular relation contains the product $\Xi_p(t)\Xi(t)$.

As shown above, our results which are true for any positive real number $p$ give, as corollaries, not only the corresponding well-known results in the theory of the Riemann zeta function when we let $p\to\infty$, but also new results when we let $p\to 0$.

\section{Some important functions and results in Koshliakov's manuscript}\label{prelim}
As mentioned at the end of Section \ref{contents}, Koshliakov studied, among other things, two generalizations of the Euler Gamma function in  his manuscript \cite[Equation (5.1), (9.1)]{koshliakov3}. He denotes these two functions by $\Gamma_{1,p}(x)$ and $\Gamma_{2,p}(x)$, and calls them the \emph{Gammamorphic functions of the first and second kind} respectively. Chapters 4 and 8 of his manuscript \cite{koshliakov3} are devoted to the study of these two functions.

The Gammamorphic function of the first kind is defined by \cite[p.~66, Chapter 4, Equation (1)]{koshliakov3}
\begin{align}
\Gamma_{1,p}(x)&:=\frac{e^{-C_{p}^{(1)}x}}{x}\prod_{j=1}^{\infty}\left\{ \frac{e^{\frac{x}{\l_j}}}{1+\frac{x}{\l_j}}\right\}^{\frac{p^2 + \l_{j}^2}{p\left(p+\frac{1}{\pi}\right)+\l^2_{j}}},\nonumber
\end{align}
where \cite[p.~46, Chapter 2, Equation (46)]{koshliakov3}
\begin{align}\label{gamma1}
C^{(1)}_{p}:=\lim_{n\to \infty}\left\{ \sum_{j=1}^{n-1}\frac{p^2+\lambda^2_{j}}{\left(p\left(p+\frac{1}{\pi}\right)+\lambda^2_j\right)} \frac{1}{\lambda_j}-\log\lambda_n\right\}
\end{align}
is Koshliakov's first generalization of the Euler constant $\g$.

The existence of the limit in \eqref{gamma1} is not explicitly given in \cite{koshliakov3}. Hence we give brief details of the same. Observe that employing the elementary evaluation $\int_{0}^{\infty}e^{-\l_j t}\, dt=1/\l_j$ and the Frullani's integral $\int_{0}^{\infty}t^{-1}\left(e^{-t}-e^{-\l_n t}\right)\, dt=\log\l_n$ in the first step below and using Lebesgue's dominated convergence theorem in the third step, we get
\begin{align*}
&\lim_{n\to \infty}\left\{ \sum_{j=1}^{n-1}\frac{p^2+\lambda^2_{j}}{\left(p\left(p+\frac{1}{\pi}\right)+\lambda^2_j\right)}\frac{1}{\l_j}-\log\lambda_n\right\}\nonumber\\
&=\lim_{n\to \infty}\left\{ \sum_{j=1}^{n-1}\frac{p^2+\lambda^2_{j}}{\left(p\left(p+\frac{1}{\pi}\right)+\lambda^2_j\right)}\int_{0}^{\infty}e^{-\l_j t}\, dt-\int_{0}^{\infty}\frac{e^{-t}-e^{-\l_n t}}{t}\, dt\right\}\nonumber\\
&=\lim_{n\to \infty}\int_{0}^{\infty}\left\{\sum_{j=1}^{n-1}\frac{p^2+\lambda^2_{j}}{\left(p\left(p+\frac{1}{\pi}\right)+\lambda^2_j\right)}e^{-\l_j t}-\frac{
(e^{-t}-e^{-\l_n t})}{t}\right\}\, dt\nonumber\\
&=\int_{0}^{\infty}\left\{\sum_{j=1}^{\infty}\frac{p^2+\lambda^2_{j}}{\left(p\left(p+\frac{1}{\pi}\right)+\lambda^2_j\right)}e^{-\l_j t}-\frac{
e^{-t}}{t}\right\}\, dt\nonumber\\
&=\int_{0}^{\infty}\left(\sigma_p(t)-\frac{e^{-t}}{t}\right)\, dt,
\end{align*}
where for $\textup{Re}(z)>0$, $\s_p(z)$ is defined by \cite[p.~44, Chapter 2, Equation (33)]{koshliakov3}
\begin{align}\label{3.33}
\s_p(z):=\sum_{j=1}^{\infty}\frac{p^2+\lambda_j^2}{p\left(p+\frac{1}{\pi}\right)+\lambda_j^2}e^{-\l_j z},
\end{align}
which, when $p\to\infty$, reduces to $1/(e^z-1)$. Now the fact that the integral $\int_{0}^{\infty}\left(\sigma_p(t)-\frac{e^{-t}}{t}\right)\, dt$ converges is easy to see since from \cite[p.~43, Chapter 2, Equation (30)]{koshliakov3}, we have
\begin{equation}
x\sigma_p(x)=1+\sum_{n=1}^{\infty}\frac{\phi_{n, p}(0)x^{n}}{n!},\nonumber
\end{equation}
where $\phi_{n, p}(0)$ are Koshliakov's generalized Bernoulli numbers. If we take $p \to \infty$, $\Gamma_{1,p}(x)$ and $C^{(1)}_{p}$ reduce to $\Gamma(s)$ and $\gamma$ respectively.

Koshliakov studied $\Gamma_{1,p}(x)$ in detail in Chapter 4 of his manuscript and derived several properties of it analogous to those of $\Gamma(x)$. The logarithmic derivative of  $\Gamma_{1,p}(x)$ is given by \cite[p.~71, Chapter 4]{koshliakov3}
\begin{align}\label{Digamma0}
\psi_{1,p}(x):=\frac{\Gamma'_{1,p}(x)}{\Gamma_{1,p}(x)}= -C_{p}^{(1)}-\frac{1}{x}+\sum_{j=1}^{\infty}\frac{p^2+\l_j^2}{p\left(p+\frac{1}{\pi} \right)+\l_j^2}\left(\frac{1}{\l_j} -\frac{1}{x+\l_j}\right).
\end{align}
In \cite[p.~71, Chapter 4, Equation (14)]{koshliakov3} he gave the following representation for it, namely,
\begin{align}\label{Digamma1}
\psi_{1,p}(x)=\left( \frac{1}{2}\frac{1}{\left(1+\frac{1}{\pi p}\right)} -1\right)\frac{1}{x}+\log x - I_1(x),
\end{align}
where $I_1(x)$ is given by \cite[p.~71, Chapter 4, Equation (15)]{koshliakov3}
\begin{align}\label{intDigamma11}
I_1(x)=\int_{0}^{\infty}\left\{\s_p(t)-\frac{1}{t}+\frac{1}{2}\frac{1}{\left(1+\frac{1}{\pi p}\right)} \right\}e^{-xt}dt,
\end{align}
as well as by \cite[p.~71, Chapter 4, Equation (16)]{koshliakov3}
\begin{align}\label{intDigamma12}
I_1(x) =2\int_{0}^{\infty}\frac{t\, dt}{(t^2+x^2)\left(\s(t)e^{2\pi t}-1\right)}.
\end{align}

The Gammamorphic function of the second kind is defined by \cite[p.~121, Chapter 8, Equation (1)]{koshliakov3}
\begin{align}
\Gamma_{2,p}(x) :=\frac{e^{-C^{(2)}_p\cdot x}}{x}\prod_{k=1}^{\infty}\frac{\exp\left(\frac{x(1,2\pi pk)_k}{k}+D_s \left\{(s,2\pi p(x+k))_k-(s,2\pi pk)_k  \right\}_{s=0} \right)}{1+\frac{x}{k}},\nonumber
\end{align}
where \cite[p.~46, Chapter 2, Equation (47)]{koshliakov3}
\begin{align}\label{gamma2}
C_p^{(2)}:=\lim_{n \to +\infty}\left\{ \sum_{k=1}^{n-1}\frac{(1,2\pi pk)_k}{k}-\frac{1}{\left(1+\frac{1}{\pi p}\right)}\log n\right\}
\end{align}
is another generalization of Euler's constant whose existence can be shown in a manner similar to that shown above for $C_{p}^{(1)}$.

In Chapter 8, Koshliakov studied the logarithmic derivative of $\Gamma_{2,p}(x)$ given by \cite[p.~124, Chapter 8, Equation (13)]{koshliakov3}
\begin{align}\label{9.13}
\psi_{2,p}(x):=\frac{\Gamma'_{2,p}(x)}{\Gamma_{2,p}(x)}=-C_p^{(2)}-\frac{1}{x}+\sum_{k=1}^{\infty}\left\{\frac{(1,2\pi pk)_k}{k} -\frac{(1,2\pi p(x+k))_k}{x+k} \right\},
\end{align}
and also gave the representation \cite[p.~124, Chapter 8, Equation (10)]{koshliakov3}: 
\begin{align}\label{Digamma2}
\psi_{2,p}(x) =-\frac{2e^{2\pi p}}{1+\frac{1}{\pi p}}Q_{2\pi p}(0)-\frac{1}{2x}+\frac{1}{\left(1+\frac{1}{\pi p}\right)}\log x -I_1^*(x),
\end{align}
where for $R (x)> 0$ the term\footnote{Koshliakov uses the notation $I_1(x)$, however, we have changed it to $I_{1}^{*}(x)$ to avoid any confusion with $I_1(x)$ defined in \eqref{intDigamma11} or \eqref{intDigamma12}.} $I_1^{*} (x)$ can be given by one of the following forms \cite[p.~124, Chapter 8, Equations (11), (12)]{koshliakov3}:
\begin{align}
I_{1}^*(x) &=\int_{0}^{\infty}\left\{ \frac{1}{\s\left(\frac{t}{2\pi} \right)e^t -1}-\frac{1}{\left(1+\frac{1}{\pi p}\right)}\frac{1}{t}+\frac{1}{2}\right\}e^{-xt}dt,\nonumber\\
I_1^*(x) &=2\int_{0}^{\infty}\frac{t}{t^2+x^2}\s_p(2\pi t)dt, \label{intDigamma22}
\end{align}
and 
\begin{align}\label{2.49}
Q_{\mu}(s)=\int_{\mu}^{\infty}e^{-t}t^{s-1}dt, \quad \mu >0,
\end{align}
is one of the two incomplete gamma functions, more commonly denoted in the contemporary literature by $\G(s,\mu)$.
%
%

We now give two modular relations of Koshliakov involving functions defined in this and the earlier section. The first one \cite[p.~150, Chapter 9, Equation (19)]{koshliakov3} given below generalizes a result of Ramanujan \cite[Equation (14)]{riemann}:

\textit{Let $ab=\pi$. Then
\begin{align}\label{10.19}
&\sqrt{a^3}\int_{0}^{\infty}xe^{-a^2x^2}\left\{\frac{1}{\s(x)e^{2\pi x}-1}+\s_p(2\pi x)-\frac{1}{2\pi}\left(1+\frac{1}{1+\frac{1}{\pi p}} \right)\frac{1}{x}\right\}\, dx\nonumber \\
&=\sqrt{b^3}\int_{0}^{\infty}xe^{-b^2x^2}\left\{\frac{1}{\s(x)e^{2\pi x}-1}+\s_p(2\pi x)-\frac{1}{2\pi}\left(1+\frac{1}{1+\frac{1}{\pi p}} \right)\frac{1}{x}\right\}\, dx\nonumber\\
&=\frac{-1}{8}\pi^{-\frac{7}{4}}\int_{0}^{\infty}\Xi_p\left(\frac{t}{2}\right)\left|\Gamma\left(-\frac{1}{4}+\frac{it}{4} \right) \right|^2\cos\left(\frac{1}{2}t\log\left(\frac{\sqrt{\pi}}{a}\right)\right)\, dt .
\end{align}}

The other modular relation \cite[p.~152--153, Chapter 9, Equations (26), (27)]{koshliakov3} is a generalization of a formula of Hardy (see \cite[Equation (1.15)]{dixthet}, \cite[Equations (14), (20)]{koshliakov5}):

\textit{Let $ab=\pi$. Then
\begin{align}\label{10.27}
&\sqrt{a}\int_{0}^{\infty}e^{-a^2x^2}\left\{\frac{\Gamma'_{1,p}(x)}{\Gamma_{1,p}(x)}+ \frac{\Gamma'_{2,p}(x)}{\Gamma_{2,p}(x)}+\frac{2e^{2\pi p}Q_{2\pi p}(0)}{1+\frac{1}{\pi p}}+\frac{2}{x}-\left(1+\frac{1}{1+\frac{1}{\pi p}}\right)\log x\right\}\, dx\nonumber \\
&=\sqrt{b}\int_{0}^{\infty}e^{-b^2x^2}\left\{\frac{\Gamma'_{1,p}(x)}{\Gamma_{1,p}(x)}+ \frac{\Gamma'_{2,p}(x)}{\Gamma_{2,p}(x)}+\frac{2e^{2\pi p}Q_{2\pi p}(0)}{1+\frac{1}{\pi p}}+\frac{2}{x}-\left(1+\frac{1}{1+\frac{1}{\pi p}}\right)\log x\right\}\, dx\nonumber\\
&=4\pi^{\frac{1}{4}}\int_{0}^{\infty}\frac{\Xi_p\left( \frac{t}{2}\right)}{t^2+1}\frac{\cos\left(\frac{1}{2}t\log\left(\frac{\sqrt{\pi}}{a}\right)\right)}{\cosh\frac{\pi t}{2}}\, dt.
\end{align}
}

\section{New results in the theory of Koshliakov zeta functions}\label{mr}

Euler's famous formula for $\zeta(2m), m\in\mathbb{N},$ is given by \cite[p.~5, Equation (1.14)]{temme}
\begin{equation}\label{ef}
\zeta(2 m ) = (-1)^{m +1} \frac{(2\pi)^{2 m}B_{2 m }}{2 (2 m)!}.
\end{equation}
Let $\s_p(z)$ be defined for $\textup{Re}(z)>0$ by \eqref{3.33}. For $k\in\mathbb{N}$, Koshliakov defines his generalized Bernoulli numbers\footnote{We note that in Koshliakov's notation, $B_{2k}^{(p)}$ would be denoted by $(-1)^{k+1}B_k^{(p)}$. We have followed the contemporary notation for Bernoulli numbers. It is easy to see that $\lim_{p\to\infty}B_{2k}^{(p)}=B_{2k}$.} by \cite[p.~46, Chapter 2, Equation~(45)]{koshliakov3}
\begin{align}\label{genber}
B_{2k}^{(p)} :=(-1)^{k+1}4k\int_{0}^{\infty}x^{2k-1}\s_p(2\pi x)dx,\hspace{1mm} B_{0}^{(p)}:=\frac{1}{1+\frac{1}{\pi p}}.
\end{align}
He then obtains a generalization of \eqref{ef}, namely \cite[Chapter 1, Equation (38)]{koshliakov3},
\begin{align}
\zeta_{p}(2m)=\frac{(-1)^{m+1}(2\pi)^{2m}}{2(2m)!}B_{2m}^{(p)}.\nonumber
\end{align}
The special values of the Riemann zeta function at even positive integers, that is, $\zeta(2m)$, are transcendental as can be seen from \eqref{ef}. However, very little is known about the arithmetic nature of the odd zeta values $\zeta(2m+1)$. Ap\'{e}ry \cite{apery1}, \cite{apery2} showed that $\zeta(3)$ is irrational. Zudilin \cite{zudilin} showed that one of $\zeta(5), \zeta(7), \zeta(9)$ and $\zeta(11)$ is irrational. There are further results, for example, \cite{ballrivoal}, eliciting the arithmetic nature of $\zeta(2m+1)$.

A celebrated formula for $\zeta(2m+1)$ is that of Ramanujan \cite[p.~173, Ch. 14, Entry 21(i)]{ramnote}, \cite[pp.~319-320, formula (28)]{lnb}, \cite[pp.~275-276]{bcbramsecnote}.  For $\a, \b>0$ with $\a\b=\pi^2$ and $m\in\mathbb{Z}, m\neq 0$, it is given by
\begin{align}\label{zetaodd}
\a^{-m}\left\{\frac{1}{2}\zeta(2m+1)+\sum_{n=1}^{\infty}\frac{n^{-2m-1}}{e^{2\a n}-1}\right\}&=(-\b)^{-m}\left\{\frac{1}{2}\zeta(2m+1)+\sum_{n=1}^{\infty}\frac{n^{-2m-1}}{e^{2\b n}-1}\right\}\nonumber\\
&\quad-2^{2m}\sum_{j=0}^{m+1}\frac{(-1)^jB_{2j}B_{2m+2-2j}}{(2j)!(2m+2-2j)!}\a^{m+1-j}\b^j,
\end{align}
where for $j\geq 0$, $B_{j}:=B_j(0)$ is the $j^{\textup{th}}$ Bernoulli number and $B_{j}(a)$ is the $j^{\textup{th}}$ Bernoulli polynomial defined by
\begin{equation*}
\sum_{j=0}^{\infty}\frac{B_j(a) z^j}{j!}=\frac{ze^{az}}{e^{z}-1} \hspace{5mm}(0<a\leq 1, |z|<2\pi). 
\end{equation*}
This formula has a rich history for which we refer the reader to \cite{berndtstraubzeta}. Two new generalizations of \eqref{zetaodd} were recently given in \cite{dgkm} and \cite{dixitmaji1}. 

\subsection{Transformations concerning $\zeta_p(s)$ at odd integer arguments}

Our first result, which is a Ramanujan-type formula for the Koshliakov zeta function $\zeta_p(s)$, is given in the following theorem.
\begin{theorem}\label{Ramanujantype1}
Let $\zeta_p(s)$ be defined in \eqref{zps} and let $\s(z)$ be defined by \eqref{sigmap}. Let $B_{2j}^{(p)}$ be defined in \eqref{genber}. Then for $m \in \mathbb{Z},~m\neq 0,$ and $\a \b =\pi^2$,
\begin{align}\label{rt1eqn}
&\a^{-m}\left\{\frac{1}{2}\zeta_{p}(2m+1)+\sum_{j=1}^{\infty}\frac{p^2+\l_j^2}{p\left(p+\frac{1}{\pi}\right)+\l_j^2}\cdot\frac{\l_{j}^{-2m-1}}{\s\left(\frac{\l_j \a}{\pi} \right)e^{2\a \l_j}-1}\right\}\nonumber\\
&=(-\b)^{-m}\left\{\frac{1}{2}\zeta_{p}(2m+1)+\sum_{j=1}^{\infty}\frac{p^2+\l_j^2}{p\left(p+\frac{1}{\pi}\right)+\l_j^2}\cdot\frac{\l_{j}^{-2m-1}}{\s\left(\frac{\l_j \b}{\pi} \right)e^{2\b \l_j}-1}\right\}\nonumber\\
&\quad -2^{2m}\sum_{j=0}^{m+1}\frac{(-1)^jB_{2j}^{(p)}B_{2m-2j+2}^{(p)}}{(2j)!(2m-2j+2)!}\a^{m-j+1}\b^j.
\end{align}
\end{theorem}
When $p\to\infty$, Theorem \ref{Ramanujantype1} reduces to Ramanujan's formula \eqref{zetaodd}. 

Even though there are several generalizations of \eqref{zetaodd} in the literature, to the best of our knowledge, there is no generalization known for the Hurwitz zeta function, defined for Re$(s)>1$ and $0<a\leq1$ by \cite[p.~36]{titch}
\begin{equation*}
\zeta(s, a)=\sum_{n=0}^{\infty}\frac{1}{(n+a)^{s}}.
\end{equation*}
When $p\to0$ in Theorem \ref{Ramanujantype1}, we get a Ramanujan-type formula for the Hurwitz zeta function $\zeta(2m+1,1/2)$.
\begin{corollary}\label{hurwitzhalf}
For $m \in \mathbb{Z},~m\neq 0$ and $\a \b =\pi^2$, we have
\begin{align*}
&\a ^{-m}\left\{\frac{1}{2}\zeta\left(2m+1,\frac{1}{2}\right)-2^{2m+1}\sum_{j=1}^{\infty}\frac{(2j-1)^{-2m-1}}{e^{(2j-1)\a }+1}\right\}\nonumber\\
&=(-\b)^{-m}\left\{\frac{1}{2}\zeta\left(2m+1,\frac{1}{2}\right)-2^{2m+1}\sum_{j=1}^{\infty}\frac{(2j-1)^{-2m-1}}{e^{(2j-1)\b }+1} \right\}\\
&\quad-2^{2m}\sum_{j=1}^{m}\frac{(-1)^jB_{2j}B_{2m-2j+2}}{(2j)!(2m-2j+2)!}(2^{2j}-1)(2^{2m-2j+2}-1)\a^{m-j+1}\b^j.
\end{align*}
\end{corollary}
Lerch's formula for $\zeta(4m+3), m\geq0$ is \cite{lerch}
\begin{align}
\zeta(4m+3)+2\sum_{j=1}^{\infty}\frac{1}{j^{4m+3}(e^{2\pi j}-1)}=\pi^{4m+3}2^{4m+2}\sum_{j=0}^{2m+2}\frac{(-1)^{j+1}B_{2j}B_{4m+4-2j}}{(2j)!(4m+4-2j)!}.\nonumber
\end{align}
If we replace $m$ by $2m+1$ and let $\a=\b=\pi$ in Theorem \ref{Ramanujantype1}, we are led to a new generalization of Lerch's theorem:
\begin{corollary}
For $m\in\mathbb{N}\cup\{0\}$ and $p>0$,
\begin{align}
\zeta_p(4m+3)+2\sum_{j=1}^{\infty}\frac{p^2+\l_j^2}{p\left(p+\frac{1}{\pi}\right)+\l_j^2}\cdot\frac{\l_{j}^{-4m-3}}{\s(\l_j)e^{2\pi \l_j}-1}=-2^{4m+2}\pi^{4m+3}\sum_{j=0}^{2m+2}\frac{(-1)^jB_{2j}^{(p)}B_{4m-2j+4}^{(p)}}{(2j)!(4m-2j+4)!}.\nonumber
\end{align}
\end{corollary}
Letting $p\to0$ in the above corollary, then employing \eqref{zpzero} and the fact that $\l_j\to j-1/2$, we get
\begin{corollary}\label{zetanew}
For $m\in\mathbb{N}\cup\{0\}$, we have
\begin{align}
&\zeta(4m+3)-\frac{2}{(1-2^{-4m-3})}\sum_{j=1}^{\infty}\frac{(2j-1)^{-4m-3}}{e^{(2j-1)\pi}+1}\nonumber\\
&=-\frac{2^{4m+2}\pi^{4m+3}}{2^{4m+3}-1}\sum_{j=1}^{2m+1}\frac{(-1)^jB_{2j}B_{4m-2j+4}}{(2j)!(4m-2j+4)!}(2^{2j}-1)(2^{4m-2j+4}-1).\nonumber
\end{align}
\end{corollary}
Substituting $m=0$ and $m=1$ in Corollary \ref{zetanew}, we are led to new formulas for $\zeta(3)$ and $\zeta(7)$ respectively:
\begin{align*}
\zeta(3)-\frac{16}{7}\sum_{j=1}^{\infty}\frac{(2j-1)^{-3}}{e^{(2j-1)\pi}+1}&=\frac{\pi^3}{28},\\
\zeta(7)-\frac{256}{127}\sum_{j=1}^{\infty}\frac{(2j-1)^{-7}}{e^{(2j-1)\pi}+1}&=\frac{7\pi^7}{22860}.
\end{align*}
Note that one cannot let $m=0$ in Theorem \ref{Ramanujantype1} because both $\zeta_p(s)$ and $\eta_p(s)$ have pole at $s=1$. However, the associated Lambert series corresponding to $m=0$ in these theorems are well-defined at $m=0$. The transformations involving them are now given in two separate theorems below.
\begin{theorem}\label{dedekindgen}
Let $p>0$ and let $\s(z)$ be defined by \eqref{sigmap}. For $\a,\b>0$ such that $\a \b =\pi^2$,
\begin{align}
\sum_{j=1}^{\infty}&\frac{p^2+\l_j^2}{p\left(p+\frac{1}{\pi}\right)+\l_j^2}\cdot\frac{\l_{j}^{-1}}{\s\left(\frac{\l_j \a}{\pi} \right)e^{2\a \l_j}-1}-\sum_{j=1}^{\infty}\frac{p^2+\l_j^2}{p\left(p+\frac{1}{\pi}\right)+\l_j^2}\cdot\frac{\l_{j}^{-1}}{\s\left(\frac{\l_j \b}{\pi} \right)e^{2\b \l_j}-1} \nonumber\\
&\qquad\qquad\qquad=\frac{1}{12}\frac{1+\frac{3}{\pi p}\left( 1+\frac{1}{\pi p}\right)}{\left( 1+\frac{1}{\pi p}\right)^3}(\b-\a)+\frac{1}{4}\log \left(\frac{\a}{\b} \right).
\end{align}
\end{theorem}
Letting $p\to\infty$ gives the well-known transformation formula for the logarithm of the Dedekind eta function \cite[Chapter 16, Entry 27(iii)]{ramnote} or \cite[p.~320, Formula (3.6)]{lnb}, namely, for $\a, \b>0$ and $\a\b=\pi^2$,
\begin{align}
\sum_{j=1}^{\infty}\frac{1}{j(e^{2j\a}-1)}-\sum_{j=1}^{\infty}\frac{1}{j(e^{2j\b}-1)}=\frac{\b-\a}{12}+\frac{1}{4}\log\left(\frac{\a}{\b}\right)\nonumber
\end{align}
whereas letting $p \to 0$ results in
\begin{corollary}
For $\a \b =\pi^2$, 
\begin{align}
\sum_{j=1}^{\infty}\frac{1}{\left(2j-1\right)}\frac{1}{e^{\a (2j-1)}+1}+\frac{1}{8}\log \a
=\sum_{j=1}^{\infty}\frac{1}{\left(2j-1\right)}\frac{1}{e^{\b (2j-1)}+1}+\frac{1}{8}\log \b.\nonumber
\end{align}
\end{corollary}
\subsection{An extension of a modular relation on page $220$ of Ramanujan's Lost Notebook}
In the middle of the page $220$ of the Lost Notebook \cite{lnb}, Ramanujan gives an elegant modular relation involving the logarithmic derivative of the gamma function $\psi(x):=\frac{\G'(x)}{\G(x)}$. This result is stated below.

\textit{Define 
\begin{align}\label{phi}
\phi(x):=\psi(x)+\frac{1}{2x}-\log x.
\end{align}
Let Riemann's functions $\xi(s)$ and $\Xi(t)$ be respectively defined by
\begin{align}\label{xiXi}
\xi(s)&:=(s-1)\pi^{-\tfrac{1}{2}s}\Gamma(1+\tfrac{1}{2}s)\zeta(s),\nonumber\\
\Xi(t)&:=\xi(\tfrac{1}{2}+it).
\end{align}
If $\a$ and $\b$ are positive numbers such that $\a\b =1$, then
\begin{align}\label{ramfor}
\sqrt{\a}&\left\{\frac{\gamma - \log(2 \pi \a)}{2\a} +\sum_{n=1}^{\infty}\phi(n\a)\right\}=\sqrt{\b}\left\{\frac{\gamma - \log(2 \pi \b)}{2\b} +\sum_{n=1}^{\infty}\phi(n\a)\right\}\nonumber\\
&=-\frac{1}{\pi^{3/2}}\int_{0}^{\infty}\left|\Xi\left( \frac{t}{2}\right) \Gamma\left(\frac{-1+it}{4} \right)\right|^2\frac{\cos\left(\frac{t}{2}\log\a \right)}{1+t^2}dt.
\end{align}}
A proof of this result can be found in \cite{bcbad} and \cite{series}. There exist several generalizations of this formula \cite[Theorem 1.4]{dixit}, \cite[Theorems 1.6, 1.7]{charram}, \cite[Theorem 8]{nkim} and \cite[Theorem 1.1.5]{dk2}.

In what follows, we obtain a new generalization of \eqref{ramfor} in the setting of Koshliakov's zeta functions $\zeta_p(s)$ and $\eta_p(s)$. 
To state it, however, we have to introduce definitions of some functions that Koshliakov gives in his manuscript.

Koshliakov defines the function $\omega_p(s)$ by \cite[p.~148, Chapter 9, Equation (8)]{koshliakov3}
\begin{align}\label{10.8}
\omega_p(s):= \frac{\zeta_p(s)+\eta_p(s)}{2}.
\end{align}
It satisfies the functional equation \cite[p.~148, Chapter 9, Equation (9)]{koshliakov3}
\begin{align}
\pi^{-\frac{s}{2}}\Gamma\left(\frac{s}{2} \right)\omega_p(s)=\pi^{-\frac{1-s}{2}}\Gamma\left(\frac{1-s}{2} \right)\omega_p(1-s).\nonumber
\end{align}
It is easy to see that
\begin{align}\label{infzero}
\lim_{p\to\infty}\omega_p(s)=\zeta(s),\hspace{5mm}\lim_{p\to0}\omega_p(s)=(2^{s-1}+2^{-s}-1)\zeta(s).
\end{align}
Koshliakov also gives generalizations of Riemann's functions defined in \eqref{xiXi}, namely \cite[p.~148, Chapter 9, Equation (10)]{koshliakov3}, 
{\allowdisplaybreaks\begin{align}
\xi_p(s)&=\frac{s(s-1)}{2}\pi^{-\frac{s}{2}}\Gamma{\left( \frac{s}{2}\right)}\omega_p(s),\label{10.10a}\\
\Xi_p(t)&=\xi_p\left(\frac{1}{2}+i t \right).\label{10.10}
\end{align}}
It is easy to check that 
\begin{equation}
\xi_p(1-s)=\xi_p(s),\nonumber
\end{equation}
and that $\Xi_p(t)$ is an even function of $t$. Also, clearly, 
\begin{align}\label{xiinfty}
\lim_{p\to\infty}\xi_{p}(s)=\xi(s),\hspace{5mm}\lim_{p\to\infty}\Xi_{p}(t)=\Xi(t).
\end{align}

We also require a lemma which gives a new integral representation for Koshliakov's second generalized Euler constant, different from that given by Koshliakov \cite[p.~48, Chapter 2, Equation (50)]{koshliakov3}.
\begin{lemma}\label{cp2new}
We have
\begin{align}
C^{(2)}_p=\int_{0}^{\infty}\left\{\frac{1}{\s\left(\frac{x}{2\pi}\right)e^{x}-1}-\frac{1}{1+\frac{1}{\pi p}}\frac{e^{-x}}{x}\right\}\, dx.\nonumber
\end{align}
\end{lemma}
Finally, we require a generalization of an integral identity of Ramanujan \cite[Equation (22)]{riemann} which we derive in the following theorem.
\begin{theorem}\label{genrama}
Let $n$ denote a positive real number. Define
\begin{align}
F_p(n):=\int_{0}^{\infty}\left|\Gamma\left(\frac{-1+it}{4}\right) \right|^{2}\Xi_{p}\left(\frac{t}{2} \right)\Xi\left(\frac{t}{2} \right)\frac{\cos(nt)}{1+t^2}dt.\nonumber
\end{align}
Then
\begin{align}\label{Fn}
F_p(n)&=\frac{1}{2}\pi^{3/2}\int_{0}^{\infty}\left( \left(\s_p\left(xe^n\right)+\frac{1}{(\s\left( \frac{xe^n}{2\pi}\right)e^{xe^{n}}-1)} \right)-\left(1+\frac{1}{1+\frac{1}{\pi p}}\right)\frac{1}{xe^{n}}\right)\nonumber\\
&\qquad\qquad\times\left(\frac{1}{e^{xe^{-n}}-1} -\frac{1}{xe^{-n}}\right)dx.
\end{align}
\end{theorem}
From \eqref{sigmap} and \eqref{3.33}, it is clear that
\begin{align*}
\lim_{p\to\infty}\sigma(z)=1, \hspace{5mm}\lim_{p\to\infty}\sigma_{p}(z)=\frac{1}{e^{z}-1}.
\end{align*}
Along with \eqref{xiinfty}, this implies that letting $p\to\infty$ in \eqref{Fn} gives Ramanujan's identity \cite[Equation (22)]{riemann}.

Armed with the definitions and notations from Section \ref{prelim} and Lemma \ref{cp2new} and Theorem \ref{genrama}, we are now ready to give our generalization of \eqref{ramfor}, which gives an infinite family of modular relations, one for each $p>0$.
\begin{theorem}\label{ramforgen}
Define 
\begin{align}\label{capitalphi}
\Phi_p(x): =\phi_{1,p}(x)+\phi_{2,p}(x),
\end{align}
where
\begin{align}
\phi_{1,p}(x):&= \psi_{1,p}(x)+\left( 1-\frac{1}{2}\frac{1}{\left(1+\frac{1}{\pi p}\right)} \right)\frac{1}{x}-\log x,\label{phi1px}\\
\phi_{2,p}(x):&=\psi_{2,p}(x)+\frac{2e^{2\pi p}}{1+\frac{1}{\pi p}}Q_{2\pi p}(0)+\frac{1}{2x}-\frac{1}{1+\frac{1}{\pi p}} \log x\label{phi2px},
\end{align}
with $\psi_{1,p}(x), \psi_{2,p}(x)$ and $Q_{\mu}(s)$ defined in \eqref{Digamma0}, \eqref{9.13} and \eqref{2.49} respectively. Let $C^{(1)}_{p}$ and $C^{(2)}_{p}$ be defined in \eqref{gamma1} and \eqref{gamma2} respectively. If $\a$ and $\b$ are positive numbers\footnote{The identity is actually valid for any complex numbers $\a$ and $\b$ such that Re$(\a)>0$ and Re$(\b)>0$.} such that $\a\b =1$, then
\begin{align}\label{ramforgeneqn}
&\sqrt{\a}\left(\frac{C^{(1)}_{p}+C^{(2)}_{p}-\left(1+\frac{1}{1+\frac{1}{\pi p}}\right)\log(2\pi \a)}{2\a}+\sum_{n=1}^{\infty}\Phi_{p}(n \a) \right)\nonumber\\
&=\sqrt{\b}\left(\frac{C^{(1)}_{p}+C^{(2)}_{p}-\left(1+\frac{1}{1+\frac{1}{\pi p}}\right)\log(2\pi \b)}{2\b}+\sum_{n=1}^{\infty}\Phi_{p}(n \b) \right)\nonumber \\
&=-\frac{2}{\pi^{3/2}}\int_{0}^{\infty}\left|\Gamma\left(\frac{-1+it}{4}\right) \right|^{2}\Xi_{p}\left(\frac{t}{2} \right)\Xi\left(\frac{t}{2} \right)\frac{\cos\left(\frac{1}{2}t\log \a \right)}{1+t^2}dt.
\end{align}
\end{theorem}
As the reader may have guessed by now, if we let $p\to\infty$ in \eqref{ramforgeneqn}, we get \eqref{ramfor}. However, letting $p\to 0$ gives an interesting new result stated below.
\begin{corollary}\label{p220p0}
Let
\begin{equation}\label{p220p0psi}
\tau(x):=\frac{1}{2}\left(\psi\left(1+\frac{x}{2} \right) -\psi\left(\frac{1+x}{2} \right) \right) +\psi\left(x+\frac{1}{2} \right) -\frac{1}{2 x}-\log x.
\end{equation}
If $\a$ and $\b$ are positive numbers such that $\a\b =1$, then
\begin{align}\label{p220p0eqn}
\sqrt{\a}&\Bigg(\frac{\gamma-\log (\pi \a)}{2\a}+\sum_{n=1}^{\infty}\tau(n\a) \Bigg) 
=\sqrt{\b}\Bigg(\frac{\gamma-\log (\pi \b)}{2\b}+\sum_{n=1}^{\infty}\tau(n\b) \Bigg) \nonumber\\
 &=-\frac{2}{\pi^{3/2}}\int_{0}^{\infty}\left|\Gamma\left( \frac{-1+it}{4}\right) \right|^{2}\Xi^{2}\left(\frac{t}{2} \right)\frac{\cos\left(\frac{1}{2}t\log \a \right)\left(\sqrt{2}\cos\left(\frac{1}{2}t\log2\right)-1\right)}{1+t^2}dt.
\end{align}
\end{corollary}

\section{Proof of Theorem \ref{Ramanujantype1} and associated results}\label{proofs}

\begin{proof}[Theorem \textup{\ref{Ramanujantype1}}][]
We prove the result only for $m\in\mathbb{Z}^{+}$. For $m\in\mathbb{Z}^{-}$, it can be similarly proved. (In fact, it is simpler to prove the result in this case.) 

We begin with stating the Mellin transform of $\frac{1}{\s(t)e^{2\pi x}-1}$ \cite[p.~32, Chapter 1, Equation (74)]{koshliakov3}, namely for $\textup{Re}(s)>1$,
\begin{align*}
\zeta_{p}(1-s)=2\cos\left(\frac{\pi s}{2}\right)\int_{0}^{\infty}\frac{x^{s-1}}{\s(x)e^{2\pi x}-1}dx,
\end{align*}
where $\sigma(z)$ is defined in \eqref{sigmap}. Then for $\textup{Re}(s)=c>1$,
\begin{align}\label{inverse1}
\frac{1}{\s\left(\frac{x}{2\pi} \right)e^{x}-1}= \frac{1}{2\pi i}\int_{c-i\infty}^{c+i\infty}\frac{\zeta_p(1-s)}{2\cos\left(\frac{\pi s}{2}\right)}\left(\frac{x}{2\pi} \right)^{-s}dx.
\end{align}
Next, let us consider
\begin{align}\label{eis1}
\sum_{j=1}^{\infty}\frac{p^2+\lambda^2_{j}}{p\left(p+\frac{1}{\pi}\right)+\lambda^2_j}\cdot \frac{\l_j^{-2m-1}}{\s\left(\frac{\l_j x}{2\pi} \right)e^{\l_j x}-1}.
\end{align}
Observe that if we let $p \to \infty$, it reduces to 
\begin{align*}
\sum_{j=1}^{\infty}\frac{j^{-2m-1}}{e^{jx}-1},
\end{align*}
which, for $m<-1$, is essentially the Eisenstein series of weight $-2m$ on $\textup{SL}_{2}\left(\mathbb{Z}\right)$. Also, letting $p \to 0$, we get  
\begin{align*}
-\sum_{j=1}^{\infty}\frac{\left(j-\frac{1}{2}\right)^{-2m-1}}{e^{\left(j-\frac{1}{2}\right)x}+1}.
\end{align*}
From \eqref{inverse1} and \eqref{eis1}, for Re$(s)=d>1,$ we have
\begin{align}\label{doublezeta}
&\sum_{j=1}^{\infty}\frac{p^2+\lambda^2_{j}}{p\left(p+\frac{1}{\pi}\right)+\lambda^2_j}\cdot \frac{\l_j^{-2m-1}}{\s\left(\frac{\l_j x}{2\pi} \right)e^{\l_j x}-1}\nonumber\\
&= \sum_{j=1}^{\infty}\frac{p^2+\lambda^2_{j}}{p\left(p+\frac{1}{\pi}\right)+\lambda^2_j}\frac{\l_j^{-2m-1}}{2\pi i}\int_{d-i\infty}^{d+i\infty}\frac{\zeta_p(1-s)}{2\cos\left(\frac{\pi s}{2}\right)}\left(\frac{\l_jx}{2\pi} \right)^{-s}dx\nonumber\\
&=\frac{1}{2\pi i}\int_{d-i\infty}^{d+i\infty}\frac{\zeta_p(1-s)\zeta_{p}(s+2m+1)}{2\cos\left(\frac{\pi s}{2}\right)}\left(\frac{x}{2\pi} \right)^{-s}dx,
\end{align}
where in the last step we interchanged the order of summation and integration because of absolute and uniform convergence and employed the series representation of $\zeta_p(s+2m+1)$ given in \eqref{zps}.

We wish to shift the line of integration from Re$(s)=d$ to Re$(s)=-d_1$, where $2m+1<d_1<2m+2$. To that end, construct the rectangular contour $ABCD$: 
$$A(d - i T), ~ B(d + i T),~ C(-d_1 + i T),~ D(-d_1 - i T).$$
The integrand has simple poles at $0,~1,~ -2m$ and $-2j-1, ~0\leq j \leq m$ all of which lie inside the contour. If $R_a$ denotes the pole of the integrand at $a$, by Cauchy's residue theorem, 
\begin{align*}
\frac{1}{2\pi i}\left( \int_{d- i\infty}^{d+i \infty}+\int_{d+ i\infty}^{-d_1+i \infty}+\int_{-d_1+ i\infty}^{-d_1-i \infty}+\int_{-d_1- i\infty}^{d-i \infty}\right)\mathfrak{H}(s)\ ds =R_{-2m}+R_0+R_1+\sum_{j=0}^{m}R_{-2j-1},
\end{align*} 
where
\begin{align*}
\mathfrak{H}(s):= \frac{\zeta_p(1-s)\zeta_{p}(s+2m+1)}{2\cos\left(\frac{\pi s}{2}\right)}\left(\frac{x}{2\pi} \right)^{-s}.
\end{align*}
Now the residues of $\mathfrak{H}(s)$ can be evaluated as given below.
\begin{align}
R_0 &=\lim_{s\to 0}s\mathfrak{H}(s) =-\frac{1}{2}\zeta_p(2m+1),\label{R1}\\
R_1 &=\lim_{s\to 1}(s-1)\mathfrak{H}(s)=\frac{1}{\left(1+\frac{1}{\pi p}\right)}\frac{1}{ x}\zeta_p(2m+2),\label{R2}\\
R_{-2m} &=\lim_{s\to -2m}(s+2m)\mathfrak{H}(s) =\frac{(-1)^m}{2}\left(\frac{x}{2\pi} \right)^{2m}\zeta_p(2m+1),\label{R3}\\
R_{-2j-1} &=\lim_{s\to -2j-1}(s+2j+1)\mathfrak{H}(s)=\frac{(-1)^j}{\pi}\left(\frac{x}{2\pi} \right)^{2j+1}\zeta_p(2+2j)\zeta_{p}(2m-2j),\label{R4}
\end{align}
where, to calculate $R_1$, we used the fact \cite[p.~22, Chapter 1, Equation (34)]{koshliakov3}
\begin{align}\label{r1a}
\zeta_p(0)=-\frac{1}{2}\frac{1}{\left(1+\frac{1}{\pi p}\right)}.
\end{align}
The fact that
\begin{align}
\int_{d+ iT}^{-d_1+iT}\mathfrak{H}(s)\ ds ~\textup{and}~\int_{-d_1- iT}^{d-iT}\mathfrak{H}(s)\ ds \to 0, ~ as ~|T| \to \infty, \nonumber
\end{align}
follows from the elementary bound \cite[p.~24, Chapter 1, Equation (43)]{koshliakov3}
\begin{align}
\zeta_{p}(\s+it) =\mathcal{O}\left(t^{\frac{1}{2}-\s} \right)\hspace{5mm}(\s<0),\nonumber
\end{align}
the functional equation \eqref{2.30} and Stirling's formula for $\G(\s+it)$ in a vertical strip $a\leq\s\leq b$ given by \cite[p.~224]{cop}
\begin{equation}
  |\Gamma(s)|=\sqrt{2\pi}|t|^{\sigma-\frac{1}{2}}e^{-\frac{1}{2}\pi |t|}\left(1+O\left(\frac{1}{|t|}\right)\right)\nonumber
\end{equation}
as $|t|\to \infty$. Hence
\begin{align}\label{cauchy}
\frac{1}{2\pi i}\int_{d- i\infty}^{d+i \infty}\mathfrak{H}(s)\ ds =R_{-2m}+R_0+R_1+\sum_{j=0}^{m}R_{-2j-1}+\frac{1}{2\pi i}\int_{-d_1- i\infty}^{-d_1+i \infty}\mathfrak{H}(s)\ ds.
\end{align}

Let us now turn to the line integral 
\begin{align*}
\int_{-d_1 -i \infty}^{-d_1 +i \infty}\frac{\zeta_p(1-s)\zeta_{p}(s+2m+1)}{2\cos\left(\frac{\pi s}{2}\right)}\left(\frac{x}{2\pi} \right)^{-s},
\end{align*}
and employ the change of variable $s \to -s-2m$ to obtain
\begin{align}
\int_{-d_1  -i \infty}^{-d_1 +i \infty}\mathfrak{H}(s)\, ds&=\left(-\frac{x^2}{4\pi^2} \right)^{m}\int_{c-i \infty}^{c+i \infty}\frac{\zeta_{p}(s+2m+1)\zeta_p(1-s)}{2\cos\left(\frac{\pi s}{2}\right)}\left(\frac{x}{2\pi} \right)^{s}\, ds,\nonumber 
\end{align}
where $1<c=\textup{Re}(s)<2.$

Again invoking the series representation of $\zeta_{p}(s+2m+1)$ and interchanging the order of summation and integration on the right-hand side of the above equation, we get
\begin{align}\label{finallininti1}
\frac{1}{2\pi i}\int_{-d_1 -i \infty}^{-d_1+i \infty}&\mathfrak{H}(s)\, ds\nonumber\\
&=\left(-\frac{x^2}{4\pi^2} \right)^{m}\sum_{j=1}^{\infty}\frac{p^2+\lambda^2_{j}}{p\left(p+\frac{1}{\pi}\right)+\lambda^2_j}\frac{\l_j^{-2m-1}}{2\pi i}\int_{c-i\infty}^{c+i\infty}\frac{\zeta_p(1-s)}{2\cos\left(\frac{\pi s}{2}\right)}\left(\frac{2\pi \l_j}{x} \right)^{-s}\, ds\nonumber\\
& =\left(-\frac{x^2}{4\pi^2} \right)^{m}\sum_{j=1}^{\infty}\frac{p^2+\lambda^2_{j}}{p\left(p+\frac{1}{\pi}\right)+\lambda^2_j}\cdot \frac{\l_j^{-2m-1}}{\s\left(\frac{2\pi \l_j}{x} \right)e^{\frac{4\pi^2\l_j}{ x}}-1},
\end{align} 
where in the last-step, we used \eqref{inverse1}.

Thus, from \eqref{doublezeta}, \eqref{R1}, \eqref{R2}, \eqref{R3}, \eqref{R4}, \eqref{cauchy} and \eqref{finallininti1}, we arrive at 
\begin{align*}
&\sum_{j=1}^{\infty}\frac{p^2+\lambda^2_{j}}{p\left(p+\frac{1}{\pi}\right)+\lambda^2_j}\cdot \frac{\l_j^{-2m-1}}{\s\left(\frac{\l_j x}{2\pi} \right)e^{\l_j x}-1}+\frac{1}{2}\zeta_p(2m+1)\nonumber\\
&=\left(-\frac{x^2}{4\pi^2} \right)^{m}\left(\sum_{j=1}^{\infty}\frac{p^2+\lambda^2_{j}}{p\left(p+\frac{1}{\pi}\right)+\lambda^2_j}\cdot \frac{\l_j^{-2m-1}}{\s\left(\frac{2\pi \l_j}{x} \right)e^{\frac{4\pi^2\l_j}{ x}}-1}+\frac{1}{2}\zeta_p(2m+1)\right)\nonumber\\
& \quad +\frac{1}{\left(1+\frac{1}{\pi p}\right)}\frac{1}{ x}\zeta_p(2m+2)+\frac{1}{\pi}\sum_{j=0}^{m}(-1)^j\zeta_p(2j+2)\zeta_{p}(2m-2j)\left(\frac{x}{2\pi} \right)^{2j+1}\nonumber\\
&=\left(-\frac{x^2}{4\pi^2} \right)^{m}\left(\sum_{j=1}^{\infty}\frac{p^2+\lambda^2_{j}}{p\left(p+\frac{1}{\pi}\right)+\lambda^2_j}\cdot \frac{\l_j^{-2m-1}}{\s\left(\frac{2\pi \l_j}{x} \right)e^{\frac{4\pi^2\l_j}{ x}}-1}+\frac{1}{2}\zeta_p(2m+1)\right)\nonumber\\
& \quad+\frac{1}{\pi}\sum_{j=0}^{m+1}(-1)^{j+1}\zeta_p(2j)\zeta_{p}(2m-2j+2)\left(\frac{x}{2\pi} \right)^{2j-1},
\end{align*}
where in the last step, we used \eqref{r1a}. Invoking \cite[p.~22, Chapter 1, Equation (38)]{koshliakov3}, 
\begin{align}
\zeta_p(2j)=\frac{(2\pi)^{2j}}{2(2j)!}(-1)^{j+1}B_{2j}^{(p)},\nonumber
\end{align}
we derive
\begin{align}
\sum_{j=1}^{\infty}&\frac{p^2+\lambda^2_{j}}{p\left(p+\frac{1}{\pi}\right)+\lambda^2_j}\cdot \frac{\l_j^{-2m-1}}{\s\left(\frac{\l_j x}{2\pi} \right)e^{\l_j x}-1}+\frac{1}{2}\zeta_p(2m+1)\nonumber\\
&=\left(-\frac{x^2}{4\pi^2} \right)^{m}\left(\sum_{j=1}^{\infty}\frac{p^2+\lambda^2_{j}}{p\left(p+\frac{1}{\pi}\right)+\lambda^2_j}\cdot \frac{\l_j^{-2m-1}}{\s\left(\frac{2\pi \l_j}{x} \right)e^{\frac{4\pi^2\l_j}{ x}}-1}+\frac{1}{2}\zeta_p(2m+1)\right)\nonumber\\
&\quad+\frac{2\pi^2(-1)^m(2\pi)^{2m}}{x}\sum_{j=0}^{m+1}(-1)^{j}\frac{B_{2j}^{(p)}B_{2m-2j+2}^{(p)}}{(2j)!(2m-2j+2)!}\left(\frac{x}{2\pi} \right)^{2j}.\nonumber
\end{align}
Now let $x=2\a$ and $\b= \pi^2/\a$ in the above equation and then multiply both sides of the resulting equation by $\a^{-m}$ so as to deduce
\begin{align}
&\a^{-m}\Bigg\{\frac{1}{2}\zeta_p(2m+1)+\sum_{j=1}^{\infty}\frac{p^2+\lambda^2_{j}}{p\left(p+\frac{1}{\pi}\right)+\lambda^2_j}\cdot \frac{\l_j^{-2m-1}}{\s\left(\frac{\l_j \a}{\pi} \right)e^{2\a\l_j }-1}\Bigg\}\nonumber\\
&=\left(-\b \right)^{-m}\left\{\frac{1}{2}\zeta_p(2m+1)+\sum_{j=1}^{\infty}\frac{p^2+\lambda^2_{j}}{p\left(p+\frac{1}{\pi}\right)+\lambda^2_j}\cdot \frac{\l_j^{-2m-1}}{\s\left(\frac{ \l_j\b}{\pi} \right)e^{2\b\l_j}-1}\right\}\nonumber\\
& \quad+2^{2m}(-1)^m\sum_{j=0}^{m+1}(-1)^{j}\frac{B_{2j}^{(p)}B_{2m-2j+2}^{(p)}}{(2j)!(2m-2j+2)!}\a^j\b^{m-j+1}.\nonumber
\end{align}
Finally replace $j$ by $m-j+1$ in the last sum on the right-hand side to complete the proof.
\end{proof}

\begin{proof}[Theorem \textup{\ref{dedekindgen}}][]
Since the proof of this result is similar to that of Theorem \ref{Ramanujantype1}, we omit most of the details. One begins by representing the first infinite series on the left-hand side as a line integral, and then note that the integrand has a double order pole at $s=0$ and simple poles as $-1$ and $1$. The rest of the proof follows by invoking the residue theorem.
\end{proof}
\begin{corollary}\label{1cor}
For $n\geq1$, and $\a\b=\pi^2$,
\begin{align*}
\a^{n+1}&\left\{\frac{1}{2}\zeta_{p}(-2n-1)+\sum_{j=1}^{\infty}\frac{p^2+\l_j^2}{p\left(p+\frac{1}{\pi}\right)+\l_j^2}\cdot\frac{\l_{j}^{2n+1}}{\s\left(\frac{\l_j \a}{\pi} \right)e^{2\a \l_j}-1}\right\}\\
&=(-\b)^{n+1}\left\{\frac{1}{2}\zeta_{p}(-2n-1)+\sum_{j=1}^{\infty}\frac{p^2+\l_j^2}{p\left(p+\frac{1}{\pi}\right)+\l_j^2}\cdot\frac{\l_{j}^{2n+1}}{\s\left(\frac{\l_j \b}{\pi} \right)e^{2\b \l_j}-1}\right\}.
\end{align*}
\end{corollary}
\begin{proof}
Replace $m$ by $-n-1, n\geq1,$ in Theorem \ref{Ramanujantype1} and note that the finite sum on the right-hand side of \eqref{rt1eqn} vanishes.
\end{proof}
The above corollary, in turn, gives the result below.
\begin{corollary}\label{2cor}
For an even positive integer $m$, we have
\begin{align}
\sum_{j=1}^{\infty}\frac{p^2+\l_j^2}{p\left(p+\frac{1}{\pi}\right)+\l_j^2}\cdot\frac{\l_{j}^{2m+1}}{\s\left(\l_j \right)e^{2\pi \l_j}-1}=-\frac{\zeta_{p}(-2m-1)}{2}.
\end{align}
\end{corollary}
\begin{proof}
Let $\a=\b=\pi$ in Corollary \ref{1cor} and let $m$ be even.
\end{proof}
If we let $p \to \infty$ in Corollary \ref{2cor}, we get Glaisher's result \cite{glaisher1889}
\begin{align}
\sum_{j=1}^{\infty}\frac{j^{2m+1}}{e^{2\pi j}-1}=-\frac{\zeta(-2m-1)}{2}=\frac{B_{2m+2}}{4m+4}\hspace{5mm}(m\hspace{1mm}\text{even}),\nonumber
\end{align}
whereas if we let $p \to 0$, we obtain Apostol's result \cite[p.~25 Exercise 15(c)]{apostolmod}:
\begin{align}
\sum_{j=1}^{\infty}\frac{(2j-1)^{2m+1}}{e^{\pi (2j-1)}+1}=2^{2m}\lim_{p\to0}\zeta_p(-2m-1)=(2^{2m+1}-1)\frac{B_{2m+2}}{4m+4}\hspace{5mm}(m\hspace{1mm}\text{even}).\nonumber
\end{align}

Note that one cannot let $m=0$ in Corollary \ref{1cor}. The corresponding result in this scenario is given below.
\begin{theorem}\label{e2gen}
For $\a\b=\pi^2,$
\begin{align}
\a&\left\{\frac{1}{2}\zeta_{p}(-1)+\sum_{j=1}^{\infty}\frac{p^2+\l_j^2}{p\left(p+\frac{1}{\pi}\right)+\l_j^2}\cdot\frac{\l_{j}}{\s\left(\frac{\l_j \a}{\pi} \right)e^{2\a \l_j}-1}\right\}\nonumber\\
&=-\b\left\{\frac{1}{2}\zeta_{p}(-1)+\sum_{j=1}^{\infty}\frac{p^2+\l_j^2}{p\left(p+\frac{1}{\pi}\right)+\l_j^2}\cdot\frac{\l_{j}}{\s\left(\frac{\l_j \b}{\pi} \right)e^{2\b \l_j}-1}\right\}-\frac{1}{4}\frac{1}{\left(1+\frac{1}{\pi p}\right)^2}.\nonumber
\end{align}
\end{theorem}
\begin{proof}
Let $m=-1$ in Theorem \ref{Ramanujantype1} and use the representation for $B_{0}^{(p)}$ from \eqref{genber}.
\end{proof}
Letting $p\to\infty$ in the above theorem leads to
\begin{align} 
\a\sum_{j=1}^{\infty}\frac{j}{e^{2j\a}-1}+\b\sum_{j=1}^{\infty}\frac{j}{e^{2j\b}-1}&=\frac{\a+\b}{24}-\frac{1}{4},\nonumber
\end{align}
which is equivalent to the transformation formula for the Eisenstein series $E_2$ on the full modular group $\textup{SL}_{2}\left(\mathbb{Z}\right)$, 

Also, if we let $p\to0$ in the above theorem, we get a nice result, namely, for $\a\b=\pi^2$,
\begin{align}
\a\left\{\frac{1}{24}-\sum_{j=1}^{\infty}\frac{2j-1}{e^{\a(2j-1)}+1}\right\}=-\b\left\{\frac{1}{24}-\sum_{j=1}^{\infty}\frac{2j-1}{e^{\b(2j-1)}+1}\right\}.\nonumber
\end{align}

\begin{corollary}\label{cor5.4}
\begin{align}
\sum_{j=1}^{\infty}\frac{p^2+\l_j^2}{p\left(p+\frac{1}{\pi}\right)+\l_j^2}\cdot\frac{\l_{j}}{\s\left(\l_j\right)e^{2\pi \l_j}-1}=-\frac{1}{2}\zeta_{p}(-1)-\frac{1}{8\pi}\frac{1}{\left(1+\frac{1}{\pi p}\right)^2}.\nonumber
\end{align}
\end{corollary}
\begin{proof}
The result follows by simply taking $\a=\b=\pi$ in Theorem \ref{e2gen}.
\end{proof}
Letting $p \to \infty$ in Corollary \ref{cor5.4}, we get 
\begin{align}
\sum_{j=1}^{\infty}\frac{j}{e^{2\pi j}-1}=\frac{1}{24}-\frac{1}{8\pi},\nonumber
\end{align}
whereas letting $p \to 0$ gives
\begin{align}
\sum_{j=1}^{\infty}\frac{2j-1}{e^{\pi (2j-1)}+1}=\frac{1}{24},\nonumber
\end{align}
using the fact $\lim_{p\to0}\zeta_p(-1)=(2^{-1}-1)\zeta(-1)=\frac{1}{24}$.

\section{Proof of Theorem \ref{ramforgen} and associated results}
We first prove a lemma which will be used in the sequel.
\begin{proof}[Lemma \textup{\ref{cp2new}}][]

Using Frullani's integral, 
\begin{align}
\log n =\int_{0}^{\infty}\frac{e^{-x}-e^{-n x}}{x}\ dx,\nonumber
\end{align}
and the identity \cite[p.~47, Chapter 2]{koshliakov3} 
\begin{align*}
\sum_{k=1}^{n-1}\frac{(1,2\pi pk)_k}{k}=\int_{0}^{\infty}\frac{\s\left(\frac{-x}{2\pi}\right)e^{-x}-\left( \s\left(\frac{-x}{2\pi}\right)\right)^ne^{-nx}}{1-\s\left(\frac{-x}{2\pi}\right)e^{-x}}\, dx,
\end{align*}
where
$$\s\left(\frac{-x}{2\pi}\right)=\frac{2\pi p-x}{2\pi p+x},$$
in \eqref{gamma2}, it is easily gleaned that  
\begin{align}
C_p^{(2)}
&=\lim_{n \to +\infty}\left\{ \int_{0}^{\infty}\frac{\s\left(\frac{-x}{2\pi}\right)e^{-x}-\left( \s\left(\frac{-x}{2\pi}\right)\right)^ne^{-nx}}{1-\s\left(\frac{-x}{2\pi}\right)e^{-x}}\, dx-\frac{1}{\left(1+\frac{1}{\pi p}\right)}\int_{0}^{\infty}\frac{e^{-x}-e^{-n x}}{x}\ dx\right\} \nonumber \\
&=\int_{0}^{\infty}\left\{\frac{1}{\s\left(\frac{x}{2\pi}\right)e^{x}-1}-\frac{1}{\left(1+\frac{1}{\pi p}\right)}\frac{e^{-x}}{x}\right\}\ dx,\nonumber
\end{align}
since $\s\left(\frac{-x}{2\pi}\right)<1$.
\end{proof}

\begin{proof}[Theorem \textup{\ref{genrama}}][]
Let $H_p(\a)$ denote the integral on the right-hand side of \eqref{Fn}, obtained after letting $n=\tfrac{1}{2}\log\a$, that is,
\begin{align}\label{hpas}
H_{p}(\a)&:=\int_{0}^{\infty}\left( \left(\s_p\left(x\sqrt{\a}\right)+\frac{1}{(\s\left( \frac{x\sqrt{\a}}{2\pi}\right)e^{x\sqrt{\a}}-1)} \right)-\left(1+\frac{1}{1+\frac{1}{\pi p}}\right)\frac{1}{x\sqrt{\a}}\right)\nonumber\\
&\qquad\qquad\times\left(\frac{1}{e^{x/\sqrt{\a}}-1} -\frac{1}{x/\sqrt{\a}}\right)dx\nonumber\\
&=\frac{2\pi}{\sqrt{\a}}\int_{0}^{\infty}\left( \left(\s_p(2\pi t)+\frac{1}{(\s(t)e^{2\pi t}-1)} \right)-\left(1+\frac{1}{1+\frac{1}{\pi p}}\right)\frac{1}{2\pi t}\right)\nonumber\\
&\qquad\qquad\times\left(\frac{1}{e^{2\pi t/\a}-1} -\frac{\a}{2\pi t}\right)\ dt,
\end{align}
where in the last step we employed the change of variable $x=2\pi t/\sqrt{\a}$. Now let
\begin{align}
f_p(t)&:=\left( \left(\s_p(2\pi t)+\frac{1}{(\s(t)e^{2\pi t}-1)} \right)-\left(1+\frac{1}{1+\frac{1}{\pi p}}\right)\frac{1}{2\pi t}\right),\nonumber\\
g(t)&:=\frac{1}{e^{2\pi t/\a}-1} -\frac{\a}{2\pi t}.\nonumber
\end{align}
If $\mathfrak{F}_p(s)$ and $\mathfrak{G}(s)$ denote the Mellin transforms of $f_p(t)$ and $g(t)$ respectively, then Parseval's identity \cite[p. 83, Equation (3.1.11)]{kp} gives
\begin{equation}\label{par}
\int_{0}^{\infty}f_p(t)g(t)\, dt=\frac{1}{2\pi i}\int_{c-i\infty}^{c+i\infty}\mathfrak{F}_p(s)\mathfrak{G}(1-s)\, ds,
\end{equation}
provided $\mathfrak{F}_p(1-c-it)\in L(-\infty, \infty)$ and $x^{c-1}g(x)\in L[0,\infty)$ and the integral on the left is absolutely convergent. To that end, note that Koshliakov \cite[p.~40, Chapter 2, Equation (25)]{koshliakov3} has shown that for Re$(s)>1$,
\begin{align*}
\s_p(2\pi t)&=\frac{1}{2\pi i}\int_{\b- i \infty}^{\b +i \infty}\Gamma(s)\zeta_{p}(s)\frac{ds}{ (2\pi t)^s},
\end{align*}
whence, for $0<c=\textup{Re}(s)<1$,
\begin{align}\label{sigmapmellin}
\s_p(2\pi t)-\frac{1}{2\pi t}&=\frac{1}{2\pi i}\int_{c-i\infty}^{c+i\infty}\Gamma(s)\zeta_{p}(s)\frac{ds}{ (2\pi t)^s}.
\end{align}
Similarly, using \cite[p.~45, Chapter 2, Equation (39)]{koshliakov3}, for $0<c=\textup{Re}(s)<1$,we have
\begin{align}\label{sigmamellin}
\frac{1}{\s(t)e^{2\pi t}-1}-\frac{1}{\left(1+\frac{1}{\pi p}\right)}\frac{1}{2\pi t}&=\frac{1}{2\pi i}\int_{\b- i \infty}^{\b +i \infty}\Gamma(s)\eta_{p}(s)\frac{ds}{ (2\pi t)^s}.
\end{align}
From \eqref{sigmapmellin} and \eqref{sigmamellin}, for $0<c=\textup{Re}(s)<1$,
\begin{align}\label{fmel}
\int_{0}^{\infty}t^{s-1}\left(\left(\s_p(2\pi t)+\frac{1}{(\s(t)e^{2\pi t}-1)} \right)-\left(1+\frac{1}{1+\frac{1}{\pi p}}\right)\frac{1}{2\pi t}\right)\, dt =\frac{2\Gamma(s)\omega_{p}(s)}{ (2\pi )^s},
\end{align}
where $\omega_p(s)$ is defined in \eqref{10.8}.

Also, from \cite[p.~23, Equation (2.7.1)]{titch},
\begin{align}\label{gmel}
\int_{0}^{\infty}t^{s-1}\left(\frac{1}{e^{2\pi t/\a}-1} -\frac{\a}{2\pi t}\right)\, dt= \left(\frac{\a}{2\pi}\right)^s\Gamma(s)\zeta(s).
\end{align}
Therefore from \eqref{hpas}, \eqref{par}, \eqref{fmel} and \eqref{gmel}, we have
{\allowdisplaybreaks\begin{align}\label{Fnl}
H_{p}(\a)&=\frac{2\pi}{\sqrt{\a}}\frac{2}{2\pi i}\int_{\frac{1}{2}-i \infty}^{\frac{1}{2}+i \infty}\left(\frac{\a}{2\pi}\right)^{1-s}\Gamma(1-s)\zeta(1-s)\frac{\Gamma(s)\omega_{p}(s)}{ (2\pi )^s}\ ds\nonumber\\
&=\frac{\sqrt{\a}}{2\pi^2 i}\int_{\frac{1}{2}-i \infty}^{\frac{1}{2}+i \infty}\Gamma\left(\frac{1-s}{2} \right)\Gamma\left(\frac{1+s}{2} \right)\Gamma\left(1-\frac{s}{2} \right)\Gamma\left(\frac{s}{2} \right)\zeta(1-s)\omega_{p}(s)\a^{-s} \ ds\nonumber\\
&=\frac{2\sqrt{\a}}{\pi^{3/2} i}\int_{\frac{1}{2}-i \infty}^{\frac{1}{2}+i \infty}\Gamma\left(1-\frac{s}{2} \right)\Gamma\left(\frac{1+s}{2} \right)\frac{\xi(1-s)\xi_{p}(s)}{(s(s-1))^2}\a^{-s} \ ds\nonumber\\
&=\frac{-\sqrt{\a}}{2\pi^{3/2} i}\int_{\frac{1}{2}-i \infty}^{\frac{1}{2}+i \infty}\Gamma\left(-\frac{s}{2} \right)\Gamma\left(-\frac{1}{2}+\frac{s}{2} \right)\frac{\xi(1-s)\xi_{p}(s)}{s(s-1)}\a^{-s} \ ds\nonumber\\
&=\frac{1}{\pi^{3/2}} \int_{-\infty}^{ \infty}\Gamma\left(\frac{-1-it}{4}\right)\Gamma\left(\frac{-1+it}{4}\right)\frac{\xi\left(\frac{1+it}{2}\right)\xi_p \left(\frac{1+it}{2}\right)}{1+t^2}\a^{\frac{-it}{2}} \ dt\nonumber\\
&=\frac{1}{\pi^{3/2}}\int_{-\infty}^{ \infty}\left|\Gamma\left(\frac{-1+it}{4}\right)\right|^2\frac{\Xi\left( \frac{t}{2}\right)\Xi_p\left( \frac{t}{2}\right)}{1+t^2}\a^{\frac{-it}{2}} \ dt\nonumber\\
&=\frac{2}{\pi^{3/2}} \int_{0}^{ \infty}\left|\Gamma\left(\frac{-1+it}{4}\right)\right|^2\Xi\left( \frac{t}{2}\right)\Xi_p\left( \frac{t}{2}\right)\frac{\cos \left(\frac{t}{2}\log \a \right)}{1+t^2} \ dt.
\end{align}}
Equation \eqref{Fn} now follows by letting $\a=e^{2n}$ in \eqref{Fnl} and then multiplying both sides of the resulting equation by $\pi^{\frac{3}{2}}/2$.
\end{proof}

\begin{proof}[Theorem \textup{\ref{ramforgen}}][]

We first show that the two series $\sum_{n=1}^{\infty}\phi_{1,p}(n \a)$ and \newline $\sum_{n=1}^{\infty}\phi_{2,p}(n \a)$ are absolutely convergent and uniformly convergent in Re$(\a)\geq\epsilon$ and Re$(\b)\geq\epsilon$ for any $\epsilon>0$.

From \eqref{Digamma1}, \eqref{intDigamma12} and \eqref{phi1px}, we have
\begin{align}\label{phi1pxint}
\phi_{1,p}(x)=-2\int_{0}^{\infty}\frac{t\, dt}{(t^2+x^2)\left(\s(t)e^{2\pi t}-1\right)}.
\end{align}
Then
\begin{align}
\left|\phi_{1,p}(x)\right|\leq\frac{2}{x^2}\int_{0}^{\infty}
\frac{t\, dt}{\left|\sigma(t)e^{2\pi t}-1\right|}.\nonumber
\end{align}
Note that the geometric series formula gives
we have
\begin{align*}
\frac{1}{\sigma(t)e^{2\pi t}-1}=\sum_{k=1}^{\infty}\left(\frac{p-t}{p+t}\right)^ke^{-2 \pi kt}.
\end{align*}
Since $\displaystyle\frac{p-t}{p+t}<1$, we have
\begin{align}
\left|\frac{1}{\sigma(t)e^{2\pi t}-1}\right|&\leq\frac{1}{e^{2\pi t}-1}\nonumber
\end{align}
whence
\begin{align}
\left|\phi_{1,p}(x)\right|\leq \frac{2}{x^2}\int_{0}^{\infty}\frac{t\, dt}{e^{2\pi t}-1}=\frac{1}{12x^2},\nonumber
\end{align}
which proves the absolute (and uniform) convergence of $\sum_{n=1}^{\infty}\phi_{1,p}(n \a)$.

Next, from \eqref{intDigamma22}, \eqref{Digamma2} and \eqref{phi2px}, we have
\begin{align}
\phi_{2,p}(x)=-2\int_{0}^{\infty}\frac{t\, dt}{t^2+x^2}\s_p(2\pi t).\nonumber
\end{align}
Employing \cite[p.~44, Chapter 2, Equation (34)]{koshliakov3}, that is, for $\textup{Re}(s)>1$,
\begin{align}\label{3.34}
\zeta_{p}(s)=\frac{1}{\Gamma(s)}\int_{0}^{\infty}x^{s-1}\s_{p}(x)dx,
\end{align}
where $\s_p(x)$ is defined in \eqref{3.33}, we obtain\footnote{From \cite[p.~15, Chapter 1]{koshliakov3}, we know that $\zeta_p(2)=\displaystyle\frac{\pi^2}{6}\frac{1+\frac{3}{\pi p}\left( 1+\frac{1}{\pi p}\right)}{\left( 1+\frac{1}{\pi p}\right)^2}$.}
\begin{align}
\left|\phi_{2,p}(x)\right|\leq\frac{2\zeta_p(2)}{x^2},\nonumber
\end{align}
since $\sigma_p(2\pi t)>0$. This proves the absolute (and uniform) convergence of $\sum_{n=1}^{\infty}\phi_{2,p}(n \a)$.

Thus we have established the absolute and uniform convergence of both $\sum_{n=1}^{\infty}\Phi_{p}(n \a)$ and $\sum_{n=1}^{\infty}\Phi_{p}(n \b)$ on $(0,\infty)$. We now proceed to prove \eqref{ramforgeneqn}. From \eqref{phi1pxint},
\begin{align}
\sum_{n=1}^{\infty}\phi_{1,p}(n \a) &=-2\sum_{n=1}^{\infty}\int_{0}^{\infty}\frac{t\, dt}{(t^2+n^2\a^2)\left(\s(t)e^{2\pi t}-1\right)} \nonumber \\
&=-\frac{2}{\a^2}\int_{0}^{\infty}\frac{t}{\left(\s(t)e^{2\pi t}-1\right)}\sum_{n=1}^{\infty}\frac{1}{(t/\a)^2+n^2}\ dt,\nonumber
\end{align}
where the interchange of the order of summation and integration can be easily justified.
Invoking the fact \cite[p.~191]{con} that for $t\neq0$, 
\begin{align}
\sum_{n=1}^{\infty}\frac{1}{t^2+4\pi^2n^2} =\frac{1}{2t}\left(\frac{1}{e^{t}-1}+\frac{1}{2}-\frac{1}{t}\right),\nonumber
\end{align}
we find that
\begin{align}\label{sumphi}
\sum_{n=1}^{\infty}\phi_{1,p}(n \a)=-\frac{2\pi}{\a}\int_{0}^{\infty}\frac{1}{\s(t)e^{2\pi t}-1}\left(\frac{1}{e^{2\pi t/\a}-1}-\frac{\a}{2\pi t}+\frac{1}{2}\right)dt.
\end{align}

Similarly, from \eqref{3.34}, we obtain
\begin{align}\label{phi2psum}
\sum_{n=1}^{\infty}\phi_{2,p}(n \a)=-\frac{2\pi}{\a}\int_{0}^{\infty}\s_p(2\pi t)\left(\frac{1}{e^{2\pi t/\a}-1}-\frac{\a}{2\pi t}+\frac{1}{2}\right)dt.
\end{align}

Using Frullani's formula to write $\log(2\pi a)$ as
\begin{align*}
\log(2\pi \a)=\int_{0}^{\infty}\frac{e^{-t/\a} -e^{-2\pi t}}{t}\ dt
\end{align*}
and applying Lemma \ref{cp2new} with $x$ replaced by $2\pi t$, we find that
\begin{align}\label{newCpFrullani}
C_p^{(2)}- \frac{1}{1+\frac{1}{\pi p}}\log(2\pi \a) =\int_{0}^{\infty}\left(\frac{2\pi}{\s(t)e^{2\pi t}-1}-\frac{1}{\left(1+\frac{1}{\pi p}\right)}\frac{e^{-t/\a} }{t}\right)\, dt.
\end{align} 
Thus from \eqref{sumphi} and \eqref{newCpFrullani}, we obtain
\begin{align}\label{newalphafisrt}
 &\sqrt{\a}\left\{\frac{C_p^{(2)}- \frac{1}{\left(1+\frac{1}{\pi p}\right)}\log(2\pi \a)}{2\a}+\sum_{n=1}^{\infty}\phi_{1,p}(n \a)\right\}\nonumber\\
&=\int_{0}^{\infty}\left(\frac{\sqrt{\a}}{t(\s(t)e^{2\pi t}-1)}-\frac{2\pi}{\sqrt{\a}\left(\s(t)e^{2\pi t}-1\right)\left(e^{2\pi t/\a}-1 \right)}-\frac{1}{1+\frac{1}{\pi p}}\frac{e^{-t/\a} }{2\sqrt{\a} t} \right)\, dt.
\end{align}

Similarly, Koshliakov's integral representation of $C_{p}^{(1)}$ \cite[p.~47, Chapter 2, Equation (48)]{koshliakov3}
\begin{align}
C^{(1)}_{p}&=\int_{0}^{\infty}\left(\s_{p}(t)-\frac{e^{-t}}{t}\right)dt, \nonumber
\end{align}
and an application of Frullani's formula 
\begin{align}
\log(2\pi \a)=\int_{0}^{\infty}\frac{e^{-t/\a} -e^{-2\pi t}}{t}\ dt,\nonumber
\end{align}
together give
\begin{align}\label{2a}
&\frac{C_p^{(1)}- \log(2\pi \a)}{2\a} =\frac{1}{2\a}\int_{0}^{\infty}\left(2\pi \s_{p}(2\pi x)-\frac{e^{-x/\a}}{x} \right) dx.
\end{align}
Hence from \eqref{phi2psum} and \eqref{2a}, we deduce 
\begin{align}\label{alphasecond}
&\sqrt{\a}\left\{\frac{C_p^{(1)}- \log(2\pi \a)}{2\a}+\sum_{n=1}^{\infty}\phi_{2,p}(n \a)\right\}\nonumber\\
&=\int_{0}^{\infty}\left(\frac{\sqrt{\a}\s_{p}(2\pi t)}{t}-\frac{2\pi \s_{p}(2\pi t)}{\sqrt{\a}\left(e^{2\pi t/\a}-1 \right)}-\frac{e^{-t/\a} }{2\sqrt{\a} t} \right)dt,
\end{align}
Now add \eqref{newalphafisrt} and \eqref{alphasecond} and use \eqref{capitalphi} to obtain
\begin{align}\label{towardsfinal}
\sqrt{\a}&\left(\frac{C^{(1)}_{p}+C^{(2)}_{p}-\left(1+\frac{1}{1+\frac{1}{\pi p}}\right)\log(2\pi \a)}{2\a}+\sum_{n=1}^{\infty}\Phi_{p}(n \a) \right)\nonumber\\
&=\int_{0}^{\infty}\Bigg(\frac{\sqrt{\alpha}}{t}\left(\s_p(2\pi t)+\frac{1}{\s(t)e^{2\pi t}-1} \right)-\frac{2\pi}{\sqrt{\alpha}\left(e^{2\pi t/\a}-1 \right)}\left(\s_p(2\pi t)+\frac{1}{\s(t)e^{2\pi t}-1} \right)\nonumber\\
&\hspace{6cm}-\left(1+\frac{1}{1+\frac{1}{\pi p}}\right)\frac{e^{-t/\a}}{2 t\sqrt{\alpha}} \Bigg)dt\nonumber\\
&=-H_p(\a)-\left(1+\frac{1}{1+\frac{1}{\pi p}}\right)G(\a),
\end{align}
where $H_p(\a)$ is given by \eqref{hpas} and 
\begin{align}
G(\a):=\int_{0}^{\infty}\left(\frac{1}{t\sqrt{\alpha}(e^{2\pi
t/\alpha}-1)}-\frac{\sqrt{\alpha}}{2\pi t^{2}}
+\frac{e^{-t/\alpha}}{2t\sqrt{\alpha}}\right)\, dt.\nonumber
\end{align}
Now from \cite[Equation (3.9)]{bcbad}, we know that $G(\a)=0$. Hence from \eqref{Fnl} and \eqref{towardsfinal}, we arrive at the equality between the extreme sides of \eqref{ramforgeneqn}. In order to get the second equality in \eqref{ramforgeneqn}, just replace $\a$ by $\b$ in this equation and note that the integral involving the Riemann $\Xi$-function is invariant because $\b=1/\a$. This completes the proof of Theorem \ref{ramforgen}.
\end{proof}

\begin{proof}[Corollary \textup{\ref{p220p0}}][]
We first show the absolute convergence of $\sum_{n=1}^{\infty}\tau(n\a)$ and $\sum_{n=1}^{\infty}\tau(n\b)$ where $\tau(x)$ is defined in \eqref{p220p0psi}.
Using the duplication formula for the psi function \cite[p.~913, Formula \textbf{8.365.6}]{grn}
\begin{align}\label{duppsi}
\psi\left(\frac{x+1}{2}\right)=-2\log 2-\psi\left(\frac{x}{2}\right)+2\psi(x),
\end{align}
twice, that is, once as it is and then with $x$ replaced by $2x$, the functional equation 
\begin{align}\label{psife}
\psi(x+1)=\psi(x)+\frac{1}{x},
\end{align}
and the well-known result \cite[p.~259, Formula 6.3.18]{as}
\begin{equation}\label{psibigo}
\psi(x)=\log x-\frac{1}{2x}+O\left(\frac{1}{x^2}\right),
\end{equation}
we deduce after considerable simplification that
$\tau(x)=O\left(\frac{1}{x^2}\right)$ implying the absolute convergence of $\sum_{n=1}^{\infty}\tau(n\a)$  and $\sum_{n=1}^{\infty}\tau(n\b)$ .

Now let $p\to0$ in Theorem \ref{ramforgen}. Note that from \eqref{gamma1},
\begin{align*}
C_{0}^{(1)}=\lim_{n\to\infty}\left\{\sum_{j=1}^{n-1}\frac{1}{j-\tfrac{1}{2}}-\log\left(n-\frac{1}{2}\right)\right\}.
\end{align*}
Now from \eqref{psife},
\begin{equation}
\psi(z+k)=\psi(z)+\sum_{j=0}^{k-1}\frac{1}{z+j}\nonumber
\end{equation}
whence, letting $z=1/2$ and $k=n-1$ implies
\begin{equation}
\sum_{j=1}^{n-1}\frac{1}{j-\tfrac{1}{2}}=\psi\left(n-\frac{1}{2}\right)-\psi\left(\frac{1}{2}\right).\nonumber
\end{equation}
Therefore
\begin{align}\label{g10}
C_{0}^{(1)}&=\lim_{n\to\infty}\left\{\psi\left(n-\frac{1}{2}\right)-\psi\left(\frac{1}{2}\right)-\log\left(\frac{1}{2}\right)\right\}\nonumber\\
&=\g+\log4.
\end{align}
where in the last step, we used \eqref{psibigo} with $x$ replaced by $n-1/2$ as well as the fact \cite[p.~914, Formula \textbf{8.366.2}]{grn}
\begin{equation}\label{psihalf}
\psi(1/2)=-\g-\log4
\end{equation}
Also, from \eqref{gamma2} and \eqref{s}, we see that
\begin{align}\label{g20}
C_{0}^{(2)}&=\sum_{k=1}^{\infty}\frac{(-1)^k}{k}=-\log2.
\end{align}
Hence from \eqref{g10}, \eqref{g20}, we see that
\begin{align*}
\frac{C^{(1)}_{0}+C^{(2)}_{0}-\lim_{p\to 0}\left(1+\frac{1}{1+\frac{1}{\pi p}}\right)\log(2\pi \a)}{2\a}=\frac{\g-\log(\pi\a)}{2\a}.
\end{align*}
Next, we show that
\begin{align}
\lim_{p\to0}\Phi_{p}(n\a)=\tau(n\a),\nonumber
\end{align}
where $\Phi_p(x)$ and $\tau(x)$ are defined in \eqref{capitalphi} and \eqref{p220p0psi} respectively. From \eqref{capitalphi}, \eqref{phi1px}, \eqref{phi2px} and \eqref{p220p0psi}, we will be done if we can show that
\begin{align}
\lim_{p\to0}\phi_{1,p}(x)&=\psi\left(x+\frac{1}{2}\right)-\log x,\label{phi1plim}\\
\lim_{p\to0}\phi_{2,p}(x)&=-\frac{1}{2x}+\frac{1}{2}\left\{\psi\left(\frac{x}{2}+1\right)-\psi\left(\frac{x+1}{2}\right)\right\}.\label{phi2plim}
\end{align}
We first prove \eqref{phi1plim}. To that end, observe that
\begin{align}\label{phi1plima}
\lim_{p\to0}\phi_{1,p}(x)=\lim_{p\to0}\psi_{1,p}(x)+\frac{1}{x}-\log x.
\end{align}
But from \eqref{Digamma0} and \eqref{g10},
\begin{align}\label{limpsi1atte}
\lim_{p\to0}\psi_{1,p}(x)=-\left(\g+\log4\right)-\frac{1}{x}+\sum_{j=1}^{\infty}\left(\frac{1}{j-\frac{1}{2}}-\frac{1}{x+j-\frac{1}{2}}\right).
\end{align}
Now employing \cite[p.~912, Formula \textbf{8.362.1}]{grn},
\begin{equation}
\psi(x)=-\gamma-\sum_{k=0}^{\infty}\left(\frac{1}{k+x}-\frac{1}{k+1}\right),\nonumber
\end{equation}
twice, once with $x=1/2$, and again with $x$ replaced with $x+1/2$ and using \eqref{psihalf}, we are led to
\begin{equation}\label{psidiff}
\sum_{j=1}^{\infty}\left(\frac{1}{j-\frac{1}{2}}-\frac{1}{x+j-\frac{1}{2}}\right)=\psi\left(x+\frac{1}{2}\right)+\g+\log4.
\end{equation}
Hence substituting \eqref{psidiff} in \eqref{limpsi1atte} proves
\begin{equation}\label{lim1phi}
\lim_{p\to0}\psi_{1,p}(x)=\psi\left(x+\frac{1}{2}\right)-\frac{1}{x}.
\end{equation}
Substituting \eqref{lim1phi} in \eqref{phi1plima} results in \eqref{phi1plim}. To prove \eqref{phi2plim}, we first show that
\begin{align}\label{firstshow}
\lim_{p\to0}\left(\psi_{2,p}(x)+\frac{2e^{2\pi p}}{1+\frac{1}{\pi p}}Q_{2\pi p}(0)\right)=-\frac{1}{x}+\frac{1}{2}\left\{\psi\left(\frac{x}{2}+1\right)-\psi\left(\frac{x+1}{2}\right)\right\}.
\end{align}
From \eqref{Digamma2} and \eqref{intDigamma22},
\begin{align}\label{midshow}
\lim_{p\to0}\left(\psi_{2,p}(x)+\frac{2e^{2\pi p}}{1+\frac{1}{\pi p}}Q_{2\pi p}(0)\right)&=-\frac{1}{2x}-2\lim_{p\to0}\int_{0}^{\infty}\frac{t\sigma_p(2\pi t)}{t^{2}+x^2}\, dt\nonumber\\
&=-\frac{1}{2x}-2\int_{0}^{\infty}\frac{t}{t^{2}+x^2}\lim_{p\to0}\sigma_p(2\pi t)\, dt,
\end{align}
where the passage of the limit through the integral can be easily justified. Now
\begin{align}
\lim_{p\to0}\sigma_p(2\pi t)=\sum_{j=1}^{\infty}\exp\left(-\left(j-\frac{1}{2}\right)2\pi t\right)=\frac{e^{\pi t}}{e^{2\pi t}-1}.\nonumber
\end{align}
Then 
\begin{align}\label{aloduppsi}
\int_{0}^{\infty}\frac{t}{t^{2}+x^2}\lim_{p\to0}\sigma_p(2\pi t)\, dt&=\int_{0}^{\infty}\frac{te^{\pi t}}{(t^{2}+x^2)(e^{2\pi t}-1)}\, dt\nonumber\\
&=\int_{0}^{\infty}\frac{t}{(t^{2}+x^2)(e^{\pi t}+1)}\, dt+\int_{0}^{\infty}\frac{t}{(t^{2}+x^2)(e^{2\pi t}-1)}\, dt\nonumber\\
&=\frac{1}{2}\left(\psi\left(\frac{x+1}{2}\right)-\log\left(\frac{x}{2}\right)\right)-\frac{1}{2}\left(\psi(x)+\frac{1}{2x}-\log x\right),
\end{align}
using the standard integral representations \cite[p.~357-358, Formulas \textbf{3.415.1, 3.415.3}]{grn}, valid for Re$(x)>0$ and Re$(\mu)>0$: 
\begin{align*}
\int_{0}^{\infty}\frac{t\, dt}{(t^2+x^2)(e^{\mu t}-1)}&=\frac{1}{2}\left\{\log\left(\frac{x\mu}{2\pi}\right)-\frac{\pi}{x\mu}-\psi\left(\frac{x\mu}{2\pi}\right)\right\}\nonumber\\
\int_{0}^{\infty}\frac{t\, dt}{(t^2+x^2)(e^{\mu t}+1)}&=\frac{1}{2}\left\{\psi\left(\frac{x\mu}{2\pi}+\frac{1}{2}\right)-\log\left(\frac{x\mu}{2\pi}\right)\right\}.
\end{align*}
Substituting \eqref{duppsi} in \eqref{aloduppsi}, employing \eqref{psife} and simplifying, we arrive at
\begin{align}\label{aftershow}
\int_{0}^{\infty}\frac{t}{t^{2}+x^2}\lim_{p\to0}\sigma_p(2\pi t)\, dt=\frac{1}{4x}+\frac{1}{4}\left\{\psi\left(\frac{x+1}{2}\right)-\psi\left(\frac{x}{2}+1\right)\right\}.
\end{align}
Now substitute \eqref{aftershow} in \eqref{midshow} to obtain \eqref{firstshow}. Finally, letting $p\to0$ in \eqref{phi2px} and invoking \eqref{firstshow} results in \eqref{phi2plim}. Along with the prior establishment of \eqref{phi1plim} and the fact that from \eqref{infzero}, \eqref{10.10a} and \eqref{10.10}, we have $\Xi_0(t)=\left(\sqrt{2}\cos\left(t\log2\right)-1\right)\Xi(t)$, this shows that the proof of \eqref{p220p0eqn} is now complete.
\end{proof}
Another proof of Corollary \ref{p220p0} using \eqref{ramfor} is now given. We emphasize, however, that it is difficult to conceive the modular transformation in Corollary \ref{p220p0} just from the corresponding modular transformation in \eqref{ramfor}. As will be seen from the proof, such a formula was conceived by us only after comparing the forms of the two integrals involving the Riemann $\Xi$-function in Corollary \ref{p220p0} and \eqref{ramfor}. Indeed, it was only after we derived Corollary \ref{p220p0} using Theorem \ref{ramforgen} that we came to know that it could be also derived from \eqref{ramfor}.

\noindent
\textbf{A second proof of Corollary \ref{p220p0}.} Note that 
\begin{align*}
&\cos\left(\frac{1}{2}t\log \a \right)\left(\sqrt{2}\cos\left(\frac{1}{2}t\log2\right)-1\right)\nonumber\\
&=\frac{1}{\sqrt{2}}\left(\cos\left(\frac{1}{2}t\log(2\a) \right)+\cos\left(\frac{1}{2}t\log\left(\frac{\a}{2}\right)\right)\right)-\cos\left(\frac{1}{2}t\log \a \right),
\end{align*}
Hence add the two resulting identities obtained by first replacing $\a$ replaced by $2\a$ in \eqref{ramfor}, then with replacing $\a$ by $\a/2$ in \eqref{ramfor}, then multiply the resultant by $1/\sqrt{2}$, and lastly subtract \eqref{ramfor} from it. This leads to
\begin{align}\label{ha0}
\mathscr{H}(\a)&=\mathscr{H}(\b)\nonumber\\
&=-\frac{1}{\pi^{3/2}}\int_{0}^{\infty}\left|\Gamma\left( \frac{-1+it}{4}\right) \right|^{2}\Xi^{2}\left(\frac{t}{2} \right)\frac{\cos\left(\frac{1}{2}t\log \a \right)\left(\sqrt{2}\cos\left(\frac{1}{2}t\log2\right)-1\right)}{1+t^2}dt,
\end{align}
where
\begin{align}\label{ha}
\mathscr{H}(\a)&:=\frac{1}{\sqrt{2}}\left\{\sqrt{2\a}\left(\frac{\gamma - \log(4 \pi \a)}{4\a}\right)+\sqrt{\frac{\a}{2}}\left(\frac{\gamma - \log(\pi \a)}{\a}\right)\right\}-\sqrt{\a}\left(\frac{\gamma - \log(2 \pi \a)}{2\a}\right)\nonumber\\
&\quad+\frac{1}{\sqrt{2}}\left\{\sqrt{2\a}\sum_{n=1}^{\infty}\phi(2n\a)+\sqrt{\frac{\a}{2}}\sum_{n=1}^{\infty}\phi\left(\frac{n\a}{2}\right)\right\}-\sqrt{\a}\sum_{n=1}^{\infty}\phi(n\a).
\end{align}
Using \eqref{phi}, we see that
\begin{align}\label{psisimp0}
&\sqrt{\a}\sum_{n=1}^{\infty}\left(\phi(2n\a)+\frac{1}{2}\phi\left(\frac{n\a}{2}\right)-\phi(n\a)\right)\nonumber\\
&=\sqrt{\a}\sum_{n=1}^{\infty}\left(\psi(2n\a)+\frac{1}{2}\psi\left(\frac{n\a}{2}\right)-\psi(n\a)+\frac{1}{4n\a}-\frac{1}{2}\log(2n\a)\right).
\end{align}
Next, employing \eqref{duppsi} in the form
\begin{align}
\psi(x)=\frac{1}{2}\left\{\psi\left(\frac{x+1}{2}\right)+\psi\left(\frac{x}{2}\right)\right\}+\log2\nonumber
\end{align}
with $x$ replaced by $2n\a$ in the first step below, we see that
\begin{align}\label{psisimp}
&\psi(2n\a)+\frac{1}{2}\psi\left(\frac{n\a}{2}\right)-\psi(n\a)\nonumber\\
&=\frac{1}{2}\left\{\psi\left(n\a+\frac{1}{2}\right)-\psi\left(n\a\right)\right\}+\log2+\frac{1}{2}\psi\left(\frac{n\a}{2}\right)\nonumber\\
&=\frac{1}{2}\psi\left(n\a+\frac{1}{2}\right)-\frac{1}{4}\psi\left(\frac{n\a+1}{2}\right)+\frac{1}{2}\log2+\frac{1}{4}\psi\left(\frac{n\a}{2}\right)\nonumber\\
&=\frac{1}{2}\psi\left(n\a+\frac{1}{2}\right)-\frac{1}{4}\psi\left(\frac{n\a+1}{2}\right)+\frac{1}{2}\log2+\frac{1}{4}\left(\psi\left(\frac{n\a}{2}+1\right)-\frac{2}{n\a}\right),
\end{align} 
where in the penultimate step, we again used \eqref{duppsi} with $x$ replaced by $n\a$ and in the ultimate step, we used the functional equation \eqref{psife}. 

Substituting \eqref{psisimp} in \eqref{psisimp0}, we arrive at
\begin{align}\label{psisimp2}
\sqrt{\a}\sum_{n=1}^{\infty}\left(\phi(2n\a)+\frac{1}{2}\phi\left(\frac{n\a}{2}\right)-\phi(n\a)\right)=\frac{1}{2}\sqrt{\a}\sum_{n=1}^{\infty}\tau(n\a),
\end{align}
where $\tau(x)$ is defined in \eqref{p220p0psi}. From \eqref{psisimp2} and \eqref{ha}, we finally see that
\begin{align}\label{ha1}
\mathscr{H}(\a)=\frac{1}{2}\sqrt{\a}\Bigg(\frac{\gamma-\log (\pi \a)}{2\a}+\sum_{n=1}^{\infty}\tau(n\a)\Bigg).
\end{align}
From \eqref{ha0} and \eqref{ha1}, we see that the proof of Corollary \ref{p220p0} is now complete.

\hfill $\square$
%

\section{Concluding remarks}\label{cr}
We have merely scratched the tip of the iceberg by obtaining a couple of new results in the theory of Koshliakov zeta functions which lay dormant in the mathematical community for about 60-70 years. Thus, obviously, there are many fundamental questions which remain open and are not only interesting but also important. We conclude our paper with some of these.

1. Koshliakov \cite[p.~22, Chapter 1, Equation (37)]{koshliakov3} has shown that for each $p\in\mathbb{R+}$ the trivial zeros of $\zeta_p(s)$ and $\eta_p(s)$ are at negative even integers $-2k$. He \cite[p.~22, Chapter 1, Equation (38)]{koshliakov3} has also generalized Euler's formula to explicitly evaluate $\zeta_p(s)$ at even positive integers, that is, 
\begin{align*}
\zeta_p(2k)=\frac{(-1)^{k+1}(2\pi)^{2k}}{2(2k)!}B^{(p)}_{2k}.
\end{align*}
Then, with the help of functional equation \eqref{2.30}, he evaluates $\eta_p(-2k-1)$ \cite[p.~22, Chapter 1, Equation (38)]{koshliakov3}:
\begin{align*}
\eta_{p}(-(2k-1))=-\frac{B^{(p)}_{2k}}{2k}.
\end{align*}
However, there is no mention of $\zeta_p(-2k-1)$ and hence also of $\eta_p(2k)$ in \cite{koshliakov3}. The difficulty in evaluating $\eta_p(2k)$ partly lies in the fact that $\eta_p(s)$ is not a Dirichlet series. On the other hand in Theorem \ref{Ramanujantype1}, we were able to obtain an analogue of Ramanujan's formula \eqref{zetaodd} for $\zeta_p(2k+1)$ but not for $\eta_p(2k+1)$. It may be interesting to try to get closed-form evaluations of $\zeta_p(-2k-1)$ and $\eta_p(2k)$ as well as attempt a Ramanujan-type formula for $\eta_p(2k+1)$, it is exists.

2. For a finite positive $p$, it would be interesting to obtain information about the non-trivial zeros of $\zeta_p(s)$ and $\eta_p(s)$. We have embarked upon such a study \cite{kzf2}.

3. In Corollary \ref{hurwitzhalf}, we obtained a Ramanujan-type formula for $\zeta\left(2m+1,\frac{1}{2}\right)$. Does there exist a full-fledged Ramanujan-type formula for $\zeta(2m+1,a)$ for any $a$ in $(0,1]$.

4. In Koshliakov's manuscript \cite{koshliakov3} as well as in this paper, we have seen that even in the case of $p\to0$, we get several new results. 

The only case where a result obtained by letting $p\to0$ in the general result is obtainable from the $p\to\infty$ case of the latter is when we are working on the critical line Re$(s)=1/2$. For example, we have seen that Theorem \ref{p220p0} can be also obtained from \eqref{ramfor} because $\Xi_0(t)=\left(\sqrt{2}\cos\left(t\log2\right)-1\right)\Xi(t)$. 

5. Multiplying Ramanujan's identity \eqref{ramfor} throughout by 2 and then adding the corresponding sides of the resulting equation to those of \eqref{p220p0eqn} gives the following result whose application is given immediately after its statement.

Let $\Omega(x)$ be defined by
$\Omega(x):=2\psi(2x)+\psi\left(\frac{x}{2}\right)+\frac{3}{2x}-3\log x-\log 2$.
If $\a$ and $\b$ are positive numbers such that $\a\b =1$, then
{\allowdisplaybreaks\begin{align}\label{omegaxeqn}
\sqrt{\a}&\left\{\frac{3\g-2\log 2-3\log(\pi\a)}{2\a}+\sum_{n=1}^{\infty}\Omega(n\a)\right\}
=\sqrt{\b}\left\{\frac{3\g-2\log2-3\log(\pi\b)}{2\b}+\sum_{n=1}^{\infty}\Omega(n\b)\right\}\nonumber\\
&=-\frac{2\sqrt{2}}{\pi^{3/2}}\int_{0}^{\infty}\left|\Gamma\left( \frac{-1+it}{4}\right) \right|^{2}\Xi^{2}\left(\frac{t}{2} \right)\frac{\cos\left(\frac{1}{2}t\log \a \right)\cos\left(\frac{1}{2}t\log2\right)}{1+t^2}dt.
\end{align}}
Now let $\a=2$ in \eqref{omegaxeqn}, so that $\b=1/2$, and observe that the resulting integral involving the Riemann $\Xi$-function is always negative. This results in the following two inequalities involving infinite series of digamma function, namely,
\begin{align*}
\sum_{n=1}^{\infty}\left\{2\psi(4n)+\psi(n)+\frac{3}{4n}-\log(16n^3)\right\}&<\frac{\log(32\pi^3)-3\g}{4},\\
\sum_{n=1}^{\infty}\left\{2\psi(n)+\psi\left(\frac{n}{4}\right)+\frac{3}{n}-\log\left(\frac{n^3}{4}\right)\right\}&<\log\left(\frac{\pi^3}{2}\right)-3\g.
\end{align*}

6. New modular relations can be obtained in the setting of Koshliakov zeta functions by studying generalizations of the integrals involving the Riemann $\Xi$-function. A plethora of such integrals in the setting of the Riemann zeta function have been studied in \cite{series}, \cite{dixthet}, \cite{dixitmoll} and \cite{drz5}. 

Further, the special cases $p\to0$ of such modular relations will be new results in the theory of Riemann zeta function. We note, for example, that letting $p\to0$ in \eqref{10.19} and \eqref{10.27} respectively lead to new results, namely, for $ab=\pi$,
\begin{align*}
&\sqrt{a^3}\int_{0}^{\infty}x e^{-a^2x^2}\left(\frac{-1}{e^{2\pi x}+1}+\frac{e^{\pi x}}{e^{2\pi x}-1}-\frac{1}{2\pi x}\right)\, dx\nonumber\\
&=\sqrt{b^3}\int_{0}^{\infty}x e^{-b^2x^2}\left(\frac{-1}{e^{2\pi x}+1}+\frac{e^{\pi x}}{e^{2\pi x}-1}-\frac{1}{2\pi x}\right)\, dx\nonumber\\
&=\frac{-1}{8}\pi^{-\frac{7}{4}}\int_{0}^{\infty}\Xi\left(\frac{t}{2}\right)\left|\Gamma\left(-\frac{1}{4}+\frac{it}{4} \right) \right|^2\cos\left(\frac{1}{2}t\log\left(\frac{\sqrt{\pi}}{a}\right)\right)\left(\sqrt{2}\cos\left(t\log2\right)-1\right)\, dt 
\end{align*}
and
\begin{align*}
\sqrt{a}\int_{0}^{\infty}&e^{-a^2x^2}\left(\tau(x)+\frac{1}{2x}\right)\, dx=\sqrt{b}\int_{0}^{\infty}e^{-b^2x^2}\left(\tau(x)+\frac{1}{2x}\right)\, dx\nonumber\\
&=4\pi^{\frac{1}{4}}\int_{0}^{\infty}\frac{\Xi\left( \frac{t}{2}\right)}{t^2+1}\frac{\cos\left(\frac{1}{2}t\log\left(\sqrt{\pi}/a\right)\right)\left(\sqrt{2}\cos\left(t\log2\right)-1\right)}{\cosh\frac{\pi t}{2}}\, dt,
\end{align*}
where $\tau(x)$ is defined in \eqref{p220p0psi}.
\begin{center}
\begin{tabular}{|c|c|c|}
\hline
Some important Functions & Generalizations in Koshliakov's manuscript & Location in manuscript \\
\hline

$\displaystyle\frac{1}{e^s-1}$& $\frac{1}{\sigma\left(\frac{s}{2\pi}\right)e^{s}-1}$& \cite[p. ~13, Equation (2)]{koshliakov3}\\

&$\sigma_{p}(s)$& \cite[p. ~44, Equation (33)]{koshliakov3}\\

\hline
$\zeta(s)$ & $\zeta_{p}(s)$ & \cite[p. ~15, Equation (11)]{koshliakov3}\\

& $\eta_{p}(s)$ & \cite[p. ~20, Equation (29)]{koshliakov3} \\

& $\omega_{p}(s):=\frac{\zeta_{p}(s)+\eta_{p}(s)}{2}$ & \cite[p. ~148, Equation (8)]{koshliakov3} \\

\hline
$\xi(s):=\pi^{-s/2}\frac{s(s-1)}{2}\Gamma\left( \frac{s}{2}\right)\zeta(s)$& $\xi_{p}(s)$ & \cite[p. ~148, Equation (10)]{koshliakov3} \\
\hline
$\Xi(t):=\xi\left( \frac{1}{2}+i t\right)$& $\Xi_{p}(t)$ & \cite[p. ~148, Equation (10)]{koshliakov3} \\

\hline
$\zeta(s,a)$ (Hurwitz zeta function) & $\zeta_{\omega}(s,a)$& \cite[p. ~62, Equation (41)]{koshliakov3}\\

& ${\eta_{\omega}(s,a)}$  & \cite[p. ~107, Equation (16)]{koshliakov3}\\
\hline
$\gamma$ (Euler constant)   &$C_{p}^{(1)}$& \cite[p. ~46, Equation (46)]{koshliakov3} \\

&$C_{p}^{(2)}$& \cite[p. ~46, Equation (47)]{koshliakov3} \\
\hline
$\Gamma(s)$    &$\Gamma_{1,p}(s)$& \cite[p. ~66, Equation (1)]{koshliakov3}\\

&$\Gamma_{2,p}(s)$  &  \cite[p. ~121, Equation (1)]{koshliakov3}\\
\hline
$\psi(x):=\Gamma'(s)/\Gamma(s)$ & $\Gamma'_{1,p}(s)/\Gamma_{1,p}(s)$ & \cite[p. ~71, Equation (14)]{koshliakov3}\\

& $\Gamma'_{2,p}(s)/\Gamma_{2,p}(s)$ & \cite[p. ~124, Equation (10)]{koshliakov3}\\

\hline
$\sum_{n=0}^{\infty}e^{-n^2\pi s}$ & $\chi_p(s)$ & \cite[p. ~153, Equation (29)]{koshliakov3}  \\
\hline
\end{tabular}
\end{center}

\begin{center}
\textbf{Acknowledgements}
\end{center}
Our first and foremost thanks go to the Center for Research Libraries (CRL), Chicago for kindly providing a scanned copy of Koshliakov's manuscript to the first author. They also sincerely thank the UIUC Math librarian Tim Cole for informing the first author about the availability of this manuscript at CRL. The first author's research was partially supported by SERB MATRICS grant MTR/2018/000251 and CRG grant CRG/2020/002367. He sincerely thanks SERB for the support. The authors really appreciate the help of Elena Peterukhina for making available the PhD thesis of A.~G.~Kisunko \cite{kisunko}. They also thank Uddipta Ghosh, Sudipta Sarkar and Somnath Gandal for interesting discussions related to the topic of the paper.

\end{document}